\documentclass[12pt,a4paper,reqno]{amsart}

\usepackage{amsmath}
\usepackage{amssymb}
\usepackage{amsthm}
\usepackage{color}
\usepackage{enumerate}
\usepackage{float}
\usepackage{wrapfig}
\usepackage{layout}
\usepackage{comment}

\newcommand{\T}{\mathbb{T}}
\newcommand{\e}{\varepsilon}
\newcommand{\F}{\mathcal{F}}

\DeclareMathOperator\USC{USC}
\DeclareMathOperator\LSC{LSC}

\DeclareMathOperator\supp{supp}

\DeclareMathOperator\Lip{Lip}

\DeclareMathOperator\SC{SC}

\def\R{\mathbb{R}}

\def\Z{\mathbb{Z}}
\def\N{\mathbb{N}}
\def\cF{\mathcal{F}_K}

\def\bye{\end{document}}
\def\by{\end{proof}\bye}

\def\hello{\begin{document}}
\def\fr{\frac} 
\def\disp{\displaystyle}  
\def\ga{\alpha}     
\def\go{\omega}
\def\gep{\varepsilon}      
\def\ep{\gep}    
\def\mid{\,:\,}   
\def\gb{\beta} 
\def\gam{\gamma}
\def\gd{\delta}
\def\gz{\zeta} 
\def\gth{\theta}   
\def\gk{\kappa} 
\def\gl{\lambda}
\def\gL{\Lambda}
\def\gs{\sigma}   
\def\gf{\varphi}                  
\def\tim{\times}                        
\def\aln{&\,}
\def\ol{\overline}
\def\ul{\underline}           
\def\pl{\partial}
\def\hb{\text}                
\def\Int{\mathop{\text{int}}}

\def\gG{\varGamma}
\def\lan{\langle}
\def\ran{\rangle}
\def\cD{\mathcal{D}}
\def\cB{\mathcal{B}}
\def\bcases{\begin{cases}}
\def\ecases{\end{cases}}
\def\balns{\begin{align*}}
\def\ealns{\end{align*}}
\def\balnd{\begin{aligned}}
\def\ealnd{\end{aligned}}
\def\1{\mathbf{1}}
\def\bproof{\begin{proof}}

\def\eproof{\end{proof}}

\def\red#1{\textcolor{red}{#1}}
\def\blu#1{\textcolor{blue}{#1}}
\def\beq{\begin{equation}}
\def\eeq{\end{equation}}
\def\bthm{\begin{theorem}}
\def\ethm{\end{theorem}}
\def\bproof{\begin{proof}}
\def\eproof{\end{proof}}

\newcommand{\lbar}[1]{\mkern 1.9mu\overline{\mkern-1.9mu#1\mkern-0.1mu}
\mkern 0.1mu}

\usepackage[scr]{rsfso}
\usepackage[abbrev]{amsrefs}
\newcommand\coolrightbrace[2]{%
\left.\vphantom{\begin{matrix} #1 \end{matrix}}\right\}#2}
\def\eqr#1{\eqref{#1}}
\def\mrm#1{\mathrm{#1}}
\def\bmat{\begin{pmatrix}}
\def\emat{\end{pmatrix}}

\newcommand{\parts}{\mathscr{P}}
\newcommand{\Pmo}{\mathcal{P}^-_{1}}
\newcommand{\Ppo}{\mathcal{P}^+_{1}}
\newcommand{\Pmk}{\mathcal{P}^-_{k}}
\newcommand{\Ppk}{\mathcal{P}^+_{k}}
\newcommand{\II}{\mathcal{I}}

\def\diag{\operatorname{diag}}
\def\rT{\mathrm{T}}
\newcommand{\Rd}{{\mathbb R}^d}
\newcommand{\leb}{\mathcal L}
\def\IN{\text{ in } }
\def\AND{\text{ and }}
\def\FOR{\text{ for }}
\def\FORALL{\text{ for all }}
\def\IF{\hb{ if }}\def\ON{\hb{ on }}
\def\I{\mathbb{I}}
\def\du#1{\lan#1\ran}
\def\bald{\begin{aligned}}
\def\eald{\end{aligned}}
\def\stm{\setminus}
\def\t{\tau}
\def\SC-{\operatorname{SC}^-}
\def\sc-{\operatorname{sc}^-}
\def\lip{\operatorname{lip}}
\def\B{\operatorname{\mathbb{B}}}
\def\PV{\mathrm{P.V.}}
\def\str{\rule{0pt}{7pt}}
\def\Sp{\operatorname{Sp}}
\def\P{\operatorname{\mathbb{P}}}
\def\erf{\eqref}
\def\cG{\mathcal{G}_K}
\def\cGprime{\mathcal{G}'_K}
\def\hx{\hat x} \def\hy{\hat y}
\def\pr{\,^\prime}
\def\gX{\Xi}
\def\cV{\mathcal{V}}
\def\AE{\text{ a.e. }}
\def\0{\mathbf{0}}

\DeclareMathOperator\Bd{B}
\DeclareMathOperator\C{C}

\usepackage{xspace,enumerate,comment,bbm,nth}

\usepackage{tikz,pgfplots}

\usepackage{hyperref}
\usepackage[normalem]{ulem}   

\usepackage[active]{srcltx}

\renewcommand{\labelitemi}{$\circ$}
\newcommand{\Ham}{\mathscr{H}(\alpha_0,\alpha_1,\gamma)}
\newcommand{\Hamper}{\mathscr{H}_{per}(\alpha_0,\alpha_1,\gamma)}
\newcommand{\NLO }{(-\Delta)^{\alpha/2}}
\newcommand{\M}{{\T^d}}
\newcommand{\Mis}{\mathfrak{M}}
\newcommand{\Sol}{\mathcal{S}}
\newcommand{\sol}{\mathfrak{S}}
\newcommand{\TM}{{\T^d\times\R^d}}
\newcommand{\TR}{{\R^d\times\R^d}}
\newcommand{\cchi}{\mbox{\large $\chi$}}
\newcommand{\eps}{\varepsilon}
\newcommand{\Dr}[1]{\mbox{\rm #1}}

\theoremstyle{definition}
\newtheorem{definition}{Definition}
\newtheorem{remark}[definition]{Remark}
\newtheorem{notation}[definition]{Notation}
\theoremstyle{plain}
\newtheorem{theorem}[definition]{Theorem}
\newtheorem{corollary}[definition]{Corollary}
\newtheorem{lemma}[definition]{Lemma}
\newtheorem{proposition}[definition]{Proposition}

\newcommand{\bluee}{\color{blue}}
\newcommand{\redd}{\color{red}}

\title[The vanishing discount problem for nonlocal HJ equations]{The vanishing discount problem for \\ 
nonlocal Hamilton-Jacobi equations}

\author[A.\ Davini \and H. Ishii]{Andrea Davini \and Hitoshi Ishii}
\address[\textsc{Hitoshi Ishii}]{Institute for Mathematics and Computer Science\newline
\indent Tsuda University  \newline
\indent   2-1-1 Tsuda-machi, Kodaira-shi, Tokyo, 187-8577 Japan.}
\email{hitoshi.ishii@waseda.jp}
\address[\textsc{Andrea Davini}]{Dipartimento di Matematica\newline \indent {Sapienza} Universit\`a di
  Roma\newline \indent P.le Aldo Moro 2, 00185 Roma Italy}
\email{davini@mat.uniroma1.it}

\begin{document}
\maketitle
\begin{abstract}
We establish a convergence result for the vanishing discount problem in the context of nonlocal HJ equations. We consider a fairly general class of discounted first-order and convex HJ equations which incorporate an integro-differential operator posed on the $d$-dimensional torus, and we show that the solutions converge to a specific critical solution as the discount factor tends to zero. 
Our approach relies on duality techniques for nonlocal convex HJ equations, building upon Hahn-Banach separation theorems to develop a generalized notion of Mather measure. 
The results are applied to a specific class of convex and superlinear Hamiltonians.
\end{abstract}

\maketitle
\section*{Introduction}

In this paper we are concerned with asymptotic behavior, as the discount factor $\gl\to 0^+$, of the solutions of an equation of the form 
\begin{equation}\label{intro eq discount}
\gl u-\II u+H (x,Du)=c \ \ \ \IN \T^d \tag{E$_\lambda$},
\end{equation}
where the Hamiltonian $H\in\C(\TM)$ is a convex in the gradient variable, $\II$ is a nonlocal operator in the 
{\em L\'evy–Ito} form, and $c$ is a real constant, hereafter termed {\em critical}, which can be characterized as the unique constant for which the following equation 
\begin{equation}\label{intro eq critical}
-\II u+H (x,Du)=c \ \ \ \IN \T^d \tag{E$_0$}
\end{equation}
may admit solutions. Here and throughout the paper, (sub-, super-) solutions are always meant in the viscosity sense, see  
\cite{users,barles,bardi}, i.e., by making use of bounded test functions with bounded derivatives up to second-order, which touch $u$ from above (subsolution) or from below (supersolution) somewhere and that satisfy \eqref{intro eq discount}, for $\gl\geq 0$, with the proper inequality at the touching point, see Section \ref{sec.nonlocal HJ} for 
more details. The nonlocal operator $\II$ is defined on such a class of test functions 
$\varphi$ as follows 
\[ 
\II \varphi (x):=\int_{\R^d}\big(\varphi (x+j(x,z))-\varphi (x)-\1_B(z)\,\du{j(x,z),D\varphi (x)}\,\big)\,\nu(dz),\quad x\in\M,
\]
where $\1_B$ is the characteristic function of the open unit ball $B$ centered at 0, and $j$ and $\nu$ are a 
{\em jump function} and a {\em L\'evy type measure}, respectively, satisfying suitable assumptions, see Section \ref{sec.notation}. An example of nonlocal operator of this form is the classical fractional Laplacian of order $s\in (0,2)$, corresponding to the case when $j(x,z)=z$ and $\nu(dz)=|z|^{-(d+s)}dz$, up to a renormalization constant. 

Our main result is the following.

\begin{theorem}\label{intro thm main}
Let $H\in\C(\TM)$ be a convex (in the momentum variable) Hamiltonian such that solutions to the equation \eqref{intro eq discount} exist, for every $\gl\in [0,1)$, and they all are $\kappa$-Lipschitz, for some constant $\kappa$ independent of $\gl$. Then the equation \eqref{intro eq discount} admits a unique continuous solution $u_\gl$ for each $\gl\in (0,1)$, and the family $\{u_\lambda\mid\gl\in (0,1)\}$ uniformly converges on $\M$, as $\gl\to 0^+$, to a distinguished solution $u_0$ of \eqref{intro eq critical}. 
\end{theorem}

The limit solution $u_0$ is furthermore identified by using {\em Mather measures}, whose definition is adapted to the present context. This is proved is Section \ref{sec.asymptotic}, see Theorem \ref{thm main}, where we have implicitly assumed $c=0$, cf. condition {\bf (EX)}. We stress that we can always reduce to this case by possibly replacing $H$ with $H-c$.

The motivation to undertake this analysis is rooted in the seminal paper \cite{LPV}, where the authors studied the existence of solutions of the following Hamilton-Jacobi (HJ) equation
\begin{equation}\label{intro hjc}
H(x,D u)=c\qquad \hbox{in $\M$},
\end{equation}
for a Hamiltonian $H\in\C(\TM)$ which is coercive in the momentum $p$, uniformly with respect to $x\in \M$. The existence of a (unique) real number $c$ for which the equation \eqref{intro hjc} does admit viscosity solutions is established in \cite{LPV}  by looking at the asymptotics, as the discount factor $\gl\to 0^+$, of the unique solution $w_\gl$ of 
\[\lambda u(x)+H(x,D u)=0 \qquad\hbox{in $\M$}.
\]
According to \cite{LPV}, the functions $-\lambda w_{\lambda}$ uniformly converge on $\M$,  as $\lambda\to 0^+$, to a constant $c$. Furthermore, the solutions $(w_\lambda)_{\lambda>0}$ are equi-Lipschitz, yielding, by the Arzel\'a-Ascoli Theorem, that the functions 
$w_{\lambda}-\min_{x\in \M}w_{\lambda}$ uniformly converge, {\em along subsequences} as $\lambda\to 0^+$, to a 
solution of \eqref{intro hjc}. This technique to show existence of pairs $(u,c)\in\C(\M)\times\R$ that solve the equation \eqref{intro hjc} is referred to as {\em ergodic approximation} and it has since been employed, {\em mutatis mutandis}, in countless settings.   
 
At that time, it was not clear whether different converging sequences yield the same limit. In fact, this issue was raised as an interesting open question in \cite[Remark II.1]{AL98}. Some constraints on possible limits were subsequently identified in \cite{Gomes,IS11}, but the breakthrough came with the work 
\cite{DFIZ1}, where the authors proved that the unique solution $u_\lambda$ of
\begin{equation*}\label{lde}   
\lambda u(x)+H(x,D_x u)=c\qquad\hbox{in $\M$}
\end{equation*}
converges to a distinguished solution of the critical equation \eqref{intro hjc}, as $\lambda\rightarrow 0^+$, under the sole additional assumption that $H$ is convex in the gradient variable. The convergence of the discounted solutions $(u_\gl)_{\gl>0}$ is proved by making use of tools issued from the weak KAM Theory and depends crucially on the convexity of the Hamiltonian. It is in particular derived from the weak convergence of a family of associated probability measures, that we shall name hereafter {\em discounted Mather measures}. When the convexity condition on $H$ is dropped, these tools can no longer be applied, and, in fact, the functions $u_\lambda$ may not converge, as it was pointed out in \cite{Z19} through a counterexample posed on the 1-dimensional torus. 

Because of the popularity raised in literature by the ergodic approximation, the work \cite{DFIZ1} garnered significant attention and opened a new area of research regarding {\em vanishing discount problems}.  The application of discounted Mather measures for the asymptotic analysis and the phenomena highlighted in \cite{DFIZ1} have since been promptly extended and generalized to a variety of different problems, such as second-order HJ equations \cite{MiTra17, IMTa17, IMTb17, Zh24}, first-order discounted HJ equations depending nonlinearly on $u$ \cite{GMT18, CCIsZh19,Chen21,WYZ21, DNYZ24}, mean field games \cite{CaPo19}, weakly coupled HJ systems \cite{DZ21, Is21, IsLi20}, first-order HJ equations on non-compact manifolds \cite{IsSic20, CDD23}.

Discounted Mather measures were introduced in \cite{DFIZ1} as occupational measures on minimizing curves, relying on the availability of variational formulas to represent the discounted solutions. 
A totally different approach for constructing both discounted and non-discounted Mather measures was developed by the second author in \cite{IMTa17, IMTb17}. This alternative method leverages functional analysis tools and duality principles. In addition to its elegance, this approach has the significant advantage of being independent of the existence of representation formulas and minimizing curves, making it a considerably more flexible tool for establishing the validity of vanishing discount convergence in various settings, see \cite{Is21, IsLi20}. The duality method employed in all these works is rooted in the Sion minimax theorem. This duality approach was further refined and simplified in \cite{IsSic20} by using the more standard Hahn-Banach separation theorems for convex subsets in locally convex Hausdorff vector spaces.

The analysis conducted in this paper further confirms the effectiveness of this approach, which is here adapted to the case 
when, alongside a convex Hamiltonian, the equation includes an integro-differential operator which accounts for long-range interactions. 

Our motivation to study the vanishing discount problem in this nonlocal setting stems from the work \cite{BLT17}, where the authors establish, among other things, a comparison principle and the Lipschitz regularity of solutions for an equation of the form \eqref{intro eq discount} posed on $\Rd$, with the possible addition of a second-order term.  With the aid of these results, the authors extend the ergodic approximation approach to this setting, recovering results analogous to the ones obtained in \cite{LPV}. 

The duality techniques that underpin our asymptotic analysis are presented in Section \ref{sec.duality}. This framework allows us to provide a generalized notion of both discounted and non-discounted Mather measures, defined as the minimizers of the integral of the Lagrangian associated with the Hamiltonian via the Fenchel transform over a suitable class of probability measures. Theorem \ref{intro thm main} is ultimately derived, as  in \cite{DFIZ1}, from the weak convergence of these discounted Mather measures to non-discounted Mather measures.

The probability measures over which the  minimization is performed have the property of being {\em closed} in a proper sense, as detailed in Appendix \ref{app.smooth subsolutions}, which generalizes the standard definition (see, for example, \cite{DFIZ1}). The aforementioned Mather measures are expected to be minimizers within this broader class of closed measures as well. The proof of this assertion relies on the existence of smooth functions that serve as approximate subsolutions to \eqref{intro eq discount}, for $\gl\geq 0$, up to an arbitrarily small fixed error. Typically, such subsolutions are obtained by smoothing an almost everywhere subsolution of \eqref{intro eq discount} with a convolution kernel. Transitioning from a (Lipschitz) viscosity subsolution $v$ of \eqref{intro eq discount} to an almost everywhere subsolution is itself a challenging issue in the current setting, since the presence of the nonlocal term requires second-order differentiability of $v$ to give a pointwise meaning to the equation. We deal with this problem by approximating  $v$ via a sup-convolution, thus producing a semi-convex approximate subsolution of \eqref{intro eq discount}, hence almost everywhere twice-differentiable by Aleksandrov's Theorem, see Appendix \ref{app.sup convolution}. The subsequent step would be to regularize this semi-convex subsolution   with a convolution kernel. Yet, this approach appears problematic to implement in this context due to the presence of a nonlocal term of such a general form. We successfully carry out this regularization by slightly strengthening one of the assumptions related to the jump function, leaving the question of whether the result remains valid in the more general case addressed herein as an open question, see Appendix \ref{app.smooth subsolutions}.\smallskip

The paper is structured as follows: Section \ref{sec.preliminaries} introduces the necessary notation and preliminaries. 
Section \ref{sec.functional HJ equations} delves into the finer aspects of nonlocal HJ equations. Section \ref{sec.duality} presents the duality techniques for nonlocal convex HJ equations that underpin our asymptotic analysis. Section \ref{sec.asymptotic} is devoted to the asymptotic analysis itself. Finally, Section \ref{sec.examples} presents a class of convex and superlinear Hamiltonians to which the established results can be applied. The paper ends with three Appendices where we prove some technical results that support our analysis.\smallskip\\
{\small
\noindent{\textsc{Acknowledgements.\,$-$}} The first author is a member of the INdAM Research Group GNAMPA. He was partially supported for this research by KAUST project /Competitive Research Grant (CRG10) 2021 – 4674-CRG2021. The second author was partially supported by the KAKENHI \#20K03688, \#23K20224, \#23K20604, \#25K07072, JSPS.
}

\numberwithin{definition}{section}
\numberwithin{equation}{section}

\section{Preliminaries}\label{sec.preliminaries}
\subsection{Notation}\label{sec.notation}
We will denote by $B_R$ and by $B$ the open ball in $\R^d$ centered at $0$ of radius $R>0$ and $1$, respectively. 
For any subset $E$ of a topological space, we denote by $\bar E$, $\Dr{int}(E)$, $\partial E$ its closure, interior and boundary, respectively. 

Given a metric space $X$, we shall denote by $\Bd(X)$, $\LSC(X)$, $\USC(X)$, $\C_b(X)$, $\C(X)$, $\Lip(X)$ the space of real functions that are, respectively, bounded, lower semicontinuous, upper semicontinuous, bounded and continuous, continuous, Lipschitz continuous on $X$. When $X$ is compact, $\C(X)$ is regarded as the topological vector space endowed with the sup-norm $\|f\|_\infty:=\sup_{x\in X}|f(x)|$. We will denote by $\C^2(M)$ the space of $C^2$ functions on $M$ equipped with the topology inducing the local uniform convergence of functions and their derivatives up to second-order, where $M$ stands either for the Euclidean space $\R^d$ or the flat d-dimensional torus $\M:=\R^d/\Z^d$. We furthermore denote by $\C_{\mathrm{b}}^2 (M)$ the space of bounded $C^2$ functions on $M$, with bounded derivatives up to second-order. Real functions defined on $\M$ will be implicitly identified with their $\Z^d$--periodic representatives defined on $\R^d$, whenever this is needed. 

Given a topological vector space $Y$, we indicate by $Y^*$ its topological dual and by $\langle\cdot,\cdot\rangle$ the pairing between $Y$ and $Y^*$. When $Y=\Rd$, 
the topological dual $Y^*$ is identified with $\Rd$ and the notation $\langle\xi,p\rangle$ will denote the scalar 
product between $\xi\in Y^*=\Rd$ and $p\in Y=\Rd$. We will  denote by $|\cdot|$ the Euclidean norm on $\Rd$.

Let $\Omega$ be a convex open subset of $\R^d$. A function $u:\Omega\to\R$ is said to be {\em semiconvex} in $\Omega$ with semiconvexity constant $C\geq 0$ (briefly, {\em $C$-semiconvex}) if\ \  $u(x)+\frac{C}{2}|x|^2$\ \ is convex in $\Omega$. A semiconvex function $u$ is locally Lipschitz in $\Omega$. Furthermore, it is almost everywhere twice differentiable in $\Omega$, namely, the following holds,  see \cite[Theorem A.2]{users}. 

\begin{theorem}[Aleksandrov's Theorem]\label{thm Alexandroff} 
Let  $\Omega$ be a convex open subset of $\R^d$ and $u$ be a semiconvex function on $\Omega$. Then 
we can find a negligible set $\Sigma_u\subset\Omega$ such that, for every $y\in\Omega\setminus\Sigma_u$, there exists a vector $p_y\in\R^d$ and a symmetric $d\times d$ matrix $B_y$ such that 
\beq\label{eq Alexandroff}
u(x)=u(y)+\lan p_y,x-y\ran +\frac12 \lan B_y(x-y),x-y\ran+o(|y-x|^2)\quad\FOR\ x\in U.
\eeq 
\end{theorem}

%
%

\subsection{Nonlocal HJ equations}\label{sec.nonlocal HJ}
We will be interested in a nonlocal Hamilton-Jacobi equation of the form 
\beq\label{eq general discounted}
\gl u-\II u+H(x,Du)=0 \ \ \IN M, \tag{HJ$^\lambda$}
\eeq
where $M$ is either the Euclidean space $\R^d$ or the flat $d$--dimensional flat torus $\M=\R^d/\Z^d$,  
$\lambda\geq 0$, $H\in\C(T^*M)$ and 
\beq\label{def I}\tag{I}
\II u(x):=\int_{\R^d}\big(u(x+j(x,z))-u(x)-\1_B(z)\,\du{j(x,z),Du(x)}\,\big)\,\nu(dz),
\quad
x\in M,
\eeq
where 
$\1_B$ is the characteristic function of the open unit ball $B$ centered at 0. Throughout the paper, we will assume $\nu$ to be a positive Borel measure 
on $\Rd$ and $j\mid M\tim \R^d \to \R^d$ to be a Borel 
measurable function. The function $j$ and the measure $\nu$ will satisfy the following standing assumptions:
\begin{itemize}
\item[\bf (N)] there exists a constant $C_\nu>0$ such that\quad $\displaystyle\int_{\R^d} 1\wedge |z|^2\,\nu(dz)\leq C_\nu$;\smallskip
\item[\bf (J1)] there exists a constant $C_j>0$ such that 
\begin{flalign*}
|j(x,z)|\leq C_j|z|,
\quad
|j(x,z)-j(y,z)|\leq C_j |x-y| |z|
\qquad  &\FORALL x, y\in M \AND z\in\R^d;
\end{flalign*}
\item[\bf (J2)] for every $R>0$, there exists a constant $C_R>0$ such that 
\begin{flalign*}
\int_{\R^d\stm B_R}|j(x,z)-j(y,z)|\nu(dz)\leq C_R|x-y|\qquad  &\FORALL x,y\in M.
\end{flalign*}
\end{itemize}

The hypotheses that we will assume on the Hamiltonian $H$ in the different sections of the paper will be chosen among the following list:
\begin{itemize}
\item[\bf (H1)]\quad $|H(x,p)-H(y,p)|\leq \omega_H\left(|x-y|(1+|p|)\right)
\qquad
\hbox{for all $x,y\in M$ and $p\in\R^d$}$\\
for some continuity modulus $\omega_H:[0,+\infty)\to [0,+\infty)$;\smallskip
\item[\bf (H2)] \quad $p\mapsto H(x,p)$\quad is convex for every fixed $x\in M$.\medskip
%
%
\end{itemize}

The nonlocal differential operator $\II$ is well defined when acting on $\C_{\mathrm{b}}^2 (M)$. In fact, the following holds. 

\begin{lemma}\label{lemma I well defined}  Assume {\bf (N),\,(J1),\,(J2)}. 
Let $\varphi\in \C_{\mathrm{b}}^2 (M)$. Then, for every $x\in M$, the function
\[
z\mapsto \varphi(x+j(x,z))-\varphi(x)- \1_B(z)\lan D\varphi(x),j(x,z)\ran\quad\hbox{belongs to $L^1(\R^d,\nu)$,}
\]
in particular the following quantity
\begin{equation*}
\II \varphi(x):=\int_{\R^d}\big(\varphi(x+j(x,z))-\varphi(x)-\1_B(z)\,\du{j(x,z),D\varphi(x)}\,\big)\,\nu(dz)
\end{equation*}
is a well defined real number. 
Furthermore, the function $x\mapsto \II\varphi(x)$ is bounded and continuous on $M$.   
\end{lemma}

\bproof Let us start with the case $M:=\R^d$. We begin by remarking that there exists a constant $C=C\left(\|\varphi\|_\infty,\,\|D^2\varphi\|_\infty\right)>0$ such that 
\begin{align*}
|\varphi(x+j(x,z))-\varphi(x)-\1_B(z)\,\langle D\varphi(x),j(x,z)\rangle|
&\leq C \left(\1_B(z)\, |j(x,z)|^2 + \1_{\R^d\setminus B}(z)\,  \right)
\end{align*}
for all $x,z\in\R^d$, in view of Taylor's Theorem and the fact that $\varphi$ is in $\C^2_b(\R^d)$.  This proves the first assertion, as well as the fact that the function $\II \varphi$ is bounded, in view of assumptions {\bf (N)} and {\bf (J1)}.

Let us proceed to prove that $\II \varphi$ is continuous on $\R^d$. For any fixed $x,y\in\R^d$ and for every $R>1$, we have 
\begin{flalign*}
&
\left|
\II \varphi(x)-\II \varphi(y)
\right|
\leq
4\|\varphi\|_\infty \nu\left(\R^d\setminus B_R\right)
+
2\|D\varphi\|_\infty\int_{B_R\setminus B} \left( |x-y|+|j(x,z)-j(y,z)| \right)\,\nu(dz)\\
&
+
\int_{B} 
\big|
\underbrace
{ 
\varphi(x+j(x,z))-\varphi(x)-\du{j(x,z),D\varphi(x)}
-
\varphi(y+j(y,z))-\varphi(y)-\du{j(y,z),D\varphi(y)}
}_{\textrm{$\Delta_{x,y}\varphi(z)$}}
\big|
\,\nu(dz),
\end{flalign*}
hence, by assumptions {\bf (N)} and {\bf (J2)}, we get 
\begin{flalign}\label{eq0 continuity I}
&\left|
\II \varphi(x)-\II \varphi(y)
\right|
\leq
4\|\varphi\|_\infty \nu\left(\R^d\setminus B_R\right)
+
2\|D\varphi\|_\infty (C_1+C_\nu)|x-y|
\!
+
\!
\int_B \! \big| \Delta_{x,y}\varphi(z)\big| 
\,\nu(dz).
\!\!\!
&&
\end{flalign}
Let us estimate the integral of $\Delta_{x,y}\varphi(z)$ over $B$. To ease notation, we will write $j_x:=j(x,z)$ and $j_y:=j(y,z)$. 
By using the integral form for the remainder in Taylor's Theorem, for every $z\in B$ we have 
\begin{flalign*}
\Delta_{x,y}\varphi(z)
&=
\int_0^1 
\Big(
\langle D^2\varphi(x+tj_x)j_x,j_x\rangle
-
\langle D^2\varphi(y+tj_y)j_y,j_y\rangle
\Big)
(1-t)
\, dt\,. 
\end{flalign*}
Setting for $(t,x,y,z)\in [0,1]\tim\R^d\tim\R^d\tim B$,
\[
g(t,x,y,z)=\left(\langle D^2\varphi(x+tj_x)j_x,j_x\rangle
-
\langle D^2\varphi(y+tj_y)j_y,j_y\rangle\right)(1-t),
\]
we see from {\bf (J1)} that for every $(t,x,y,z)\in[0,1]\tim\R^d\tim\R^d\tim B$,
\[
|g(t,x,y,z)|\leq \|D^2\varphi\|_\infty(|j(x,z)|^2+|j(y,z)|^2)\leq 2C_j\|D^2\varphi\|_\infty|z|^2, 
\]
and for each fixed $(t,x,z)\in [0,1]\tim\R^d\tim B$,
\[
\lim_{y\to x} g(t,x,y,z)=g(t,x,x,z)=0.
\]
Hence, by assumption {\bf (N)} and the Dominated Convergence Theorem, we deduce that 
\beq\label{eq2 continuity I}
\lim_{y\to x} 
\int_B \Delta_{x,y}\varphi(z)\,\nu(dz) 
=\lim_{y\to x} \int_B\int_0^1 g(t,x,y,z)\,dt\,\nu(dz)=0.
\eeq
By passing to the limit in 
\eqref{eq0 continuity I} for $y\to x$, we infer from  
\eqref{eq2 continuity I} above that 
\[
\limsup_{y\to x} 
\left|
\II \varphi(x)-\II \varphi(y)
\right|
\leq
4\|\varphi\|_\infty \nu\left(\R^d\setminus B_R\right)
\quad
\FORALL R>1.
\]
Now we send $R\to+\infty$ and derive from {\bf (N)} that $\II \varphi$ is continuous at $x$, for every $x\in\R^d$.

In the periodic case $M:=\T^d$, it is easily seen that the function $\II \varphi$ is $\Z^d$--periodic on $ M$, namely it belongs to $\C(\M)$. 
\eproof

Given a function $u\in\Bd( M)$, we will denote by $u^*,u_*$ its upper and lower semicontinuous envelope, i.e. 
\beq\label{def semicontinuous envelope}
u^*(x):=\inf_{\rho>0} \sup_{|y-x|<\rho} u(y),
\qquad
u_*(x):=\sup_{\rho>0} \inf_{|y-x|<\rho} u(y)
\qquad
\FORALL x\in M.
\eeq
We give the definition of discontinuous sub and supersolution to \eqref{eq general discounted}.

\begin{definition} \label{def1}Let $u\in\Bd( M)$. 
\begin{itemize}
\item[\em (i)] We will say that $u$ is a {\em subsolution}  
of \eqref{eq general discounted} if, whenever $\varphi\in \C^2_{\mathrm{b}}( M)$ satisfies 
$(u^*-\varphi)(\hat x)=\max(u^*-\varphi)$ at some point $\hat x\in  M$, we have
\[
\gl u^*(\hx)-\II\varphi(\hat x)+H(\hat x,D\varphi(\hat x))\leq 0. 
\] 
\item[\em (ii)] We will say that $u$ is a {\em supersolution}  
of \eqref{eq general discounted} if, whenever $\varphi\in \C^2_{\mathrm{b}}( M)$ satisfies 
$(u_*-\varphi)(\hat x)=\min(u_*-\varphi)$ at some point $\hat x\in  M$, we have
\[
\gl u_*(\hx)-\II\varphi(\hat x)+H(\hat x,D\varphi(\hat x))\geq 0. 
\]
\end{itemize} 

Finally, we will say that $u\in\Bd( M)$ is a {\em solution} of \eqref{eq general discounted} if it is both a subsolution and a supersolution.
\end{definition}

Next, we state the following result, which can be proved along the same lines of \cite[Proof of Proposition 2.11]{bardi}.

\begin{proposition}\label{prop sup inf} Let $U$ be an open subset of $M$.
\begin{itemize}
\item[\em (a)] Let $\Sol^-$ be a set of functions such that $v^*$ is a subsolution of \eqref{eq general discounted} in $U$ for all $v\in\Sol^-$, and define
\[
u(x):= \sup_{v\in\Sol^-} v(x),\qquad\hbox{$x\in U$}
\]
If $u$ is locally bounded, then $u$ is a subsolution of \eqref{eq general discounted} in $U$.
\item[\em (b)] Let $\Sol^+$ be a set of functions such that $w_*$ is supersolution of \eqref{eq general discounted} in $U$ for all
$w\in\Sol^+$, and define
\[
u(x):= \inf_{w\in\Sol^+} w(x),\qquad\hbox{$x\in U$}.
\]
If $u$ is locally bounded, then $u_*$ is a supersolution of \eqref{eq general discounted} in $U$. 
\end{itemize}
\end{proposition}

\section{More on nonlocal HJ  equations}\label{sec.functional HJ equations}

\subsection{Generalities}
Let us consider the following nonlocal Hamilton-Jacobi equation 
\beq\label{eq discounted}
\gl u-\II u+H(x,Du)=0 \ \ \IN \R^d,
\tag{HJ$^\lambda_{\R^d}$}
\eeq
with $\lambda\geq 0$ and $H\in\C(\TR)$. Throughout this section, we will assume conditions {\bf (N),\,(J1),\,(J2)} with $M:=\R^d$.

We start by providing some equivalent ways to check the subsolution test for a bounded upper semicontinuous function. 
\begin{theorem}\label{thm sub equivalent}
Let $\lambda\geq 0$. Assume $H\in\C(\TR)$ and {\bf (N),\,(J1),\,(J2)}. 
Let $u\in\Bd(\R^d)\cap\USC(\R^d)$ and let  $\varphi\in \C^2_{{b}}(\R^d)$ be such that 
$(u-\varphi)(\hat x)=\max_{\R^d}(u-\varphi)$ at some point $\hat x\in \R^d$. The following are equivalent facts:
\begin{itemize}
\item[\em (i)] \quad $\gl u(\hx)-\II\varphi(\hat x)+H(\hat x,D\varphi(\hat x))\leq 0$;
\item[\em (ii)] \quad for every $\delta\in (0,1)$, 
\[
\gl u(\hx)-\II_\gd^1\varphi(\hat x)
-\II_\gd^2 (u,D\varphi(\hat x),\hat x)
+H(\hat x,D\varphi(\hat x))\leq 0,
\] 
where 
\begin{align}
& \II_\gd^1\varphi(x):=\int_{B_\gd} \big(\varphi(x+j(x,z))-\varphi(x)-\du{j(x,z),D\varphi(x)}\big)\,\nu(dz), \label{def11}
\\&\II_\gd^2( u, p,x):=\int_{\R^d\stm B_\gd} \big(u(x+j(x,z))-u(x)-\1_B(z)\du{j(x,z),p}\big)\,\nu(dz). \label{def12}
\end{align}
\item[\em(iii)] \quad $
\gl u(\hx)-\II(u,D\varphi(\hx),\hat x)+H(\hat x,D\varphi(\hat x))\leq 0,
$
where 
\[
{\II (u,p,x)}:=\int_{\R^d}\big(u(x+j(x,z))-u(x)-\1_B(z)\lan p,j(x,z)\ran\big)\,\nu(dz);
\]
\end{itemize}
Furthermore, if either one of the above facts holds, then 
\beq\label{eq inequality I}
\II(u,D\varphi(\hx),\hx)\leq \II_\gd^1\varphi(\hx)+\II_\gd^2(u,D\varphi(\hx),\hx)\leq 
\II\varphi(\hx)
\eeq
and each term in this inequality is finite. 
\end{theorem}

\bproof
As usual, we may assume that $(u-\varphi)(\hx)=0$.  
Since 
\[
u(x)\leq \varphi(x) \ \ \FOR x\in\R^d\quad \AND\quad u(\hx)=\varphi(\hx),
\]
we see that claim \eqref{eq inequality I} holds. 
From this, it follows that {\em (iii)}\,$\Rightarrow$\,{\em (ii)}\,$\Rightarrow$\,{\em (i)}.

Now, let us assume that {\em (i)} holds. Thanks to Lemma \ref{approx}, there is a sequence $\{\psi_k\mid k\in\N\}\subset \C_{\mrm{b}}^2(\R^d)$ 
such that 
\[
\psi_k(x)\geq \psi_{k+1}(x) \ \ \FOR (x,k)\in\R^d\tim\N \ \ 
\AND \ \ \lim_k \psi_k(x)=u(x) \ \ \FOR x\in\R^d.
\]
Notice that $\psi_k(x)\geq u(x)$ for $(x,k)\in\R^d\tim\N$.
Fix a function $\chi\in \C^\infty_c(\R^d)$ such that 
$0\leq\chi\leq 1$ in $\R^d$ and $\chi(x)=1$ in an open neighborhood $U$ of $\hx$.  We set
\[
\varphi_k(x)=(1-\chi(x))\psi_k(x)+\chi(x)\varphi(x) \ \ \FOR (x,k)\in\R^d\tim\N.
\]
Note that $\varphi_k(x)=\varphi(x)$ for $(x,k)\in U\tim\N$, which implies that 
$\varphi_k(\hx)=u(\hx)$, $D\varphi_k(\hx)=D\varphi(\hx)$
and $\varphi_k(x)\geq u(x)$.  By {\em (i)}, we have
\[
\gl u(\hx)+H(\hx,D\varphi(\hx))\leq \II\varphi_k(\hx) \ \ \FORALL k\in\N,
\]
which implies that 
\[
\inf_k \II\varphi_k(\hx)\geq \gl u(\hx)+H(\hx,D\varphi(\hx))>-\infty.
\]
Since 
\[
\II\varphi_k({\hx})=\int_{\R^d} \big(\varphi_k(\hx+j(\hx,z))-\varphi(\hx)-\1_B(z)\lan D\varphi(\hx),j(\hx,z)\ran\big)\,\nu(dz),
\]
\[
\varphi_k(x)\geq \varphi_{k+1}(x) \ \ \FORALL (x,k)\in\R^d\tim\N,\quad \lim_k \varphi_k(x)=((1-\chi)u)(x)+(\chi \varphi)(x) \ \ \FORALL x\in\R^d,
\]
and, by Lemma \ref{lemma I well defined}, 
\[
z\mapsto \varphi_k(\hx+j(\hx,z))-\varphi(\hx)-\1_B(z)\lan D\varphi(\hx),j(\hx,z)\ran \in L^1(\R^d,\nu),
\]
we see by the Monotone Convergence Theorem that
\begin{flalign} \label{eq1.3}
&\gl u(\hx)+H(\hx,D\varphi(\hx))\\
&
\leq 
\int_{\R^d}\Big( \big( (1-\chi) u\big)(\hx+j(\hx,z))+(\chi\varphi)(\hx+j(\hx,z))
-\varphi(\hx)-\1_B(z)\lan D\varphi(\hx),j(\hx,z)\ran\Big)\,\nu(dz).\nonumber
\end{flalign}
We choose a sequence $\{\chi_j\mid j\in\N\} \subset C^\infty(\R^d)$ such that 
\[\left\{\bald
&0\leq \chi_j\leq 1 \ \ \ON \R^d, \ \ \chi_j=1 \text{ in a neighborhood of }\hx, \ \ 
\\& \chi_j\geq \chi_{j+1} 
\ \ON \R^d \ \FOR j\in\N, \ \ 
\\&\lim_{j}\chi_j(x)=\bcases 0 &  \FOR x\not=\hx, \\
1 & \FOR x=\hx.
\ecases
\eald\right.
\]
Noting that, as $k\to +\infty$,  
\begin{flalign*}
\big((1-\chi_k)u\big)(x)+(\chi_k\varphi)(x) 
&
=
\big((1-\chi_{k+1})u\big)(x)+(\chi_{k+1}\varphi)(x)
+(\chi_k-\chi_{k+1})(\varphi-u)(x)
\\
&\geq ((1-\chi_{k+1})u)(x) +(\chi_{k+1}\varphi)(x) 
\\
&=
u(x)+(\chi_{k+1}(\varphi-u))(x)
\to u(x) \ \ \ \FOR x\in \R^d,
\end{flalign*}
by the Monotone Convergence Theorem we deduce from \eqr{eq1.3} that 
\[
z\mapsto u(\hx+j(\hx,z))-u(\hx)-\1_B(z)\lan D\varphi(\hx),j(\hx,z)\ran \in L^1(\R^d,\nu),
\]
and 
\[
\gl u(\hx)+H(\hx,D\varphi(\hx))\leq I(u,D\varphi(\hx),\hx),
\]
that is, {\em (iii)} holds. This shows that  {\em (i)}\,$\Rightarrow$\,{\em (iii)}.

Last, $\II \varphi(\hx)\in\R$ by Proposition \ref{lemma I well defined} and 
$\II(u,D\varphi(\hx),\hat x)>-\infty$ because it satisfies the inequality in item {\em (iii)}. 
This concludes the proof. 
\eproof

An analogous statement holds for the supersolution test. The result is the following.

\begin{theorem}\label{thm super equivalent}
Assume {\bf (N),\,(J1),\,(J2)}.  
Let $w\in\Bd(\R^d)\cap\LSC(\R^d)$ and let  $\varphi\in \C^2_{b}(\R^d)$ be such that 
$(w-\varphi)(\hat x)=\min(w-\varphi)$ at some point $\hat x\in \R^d$. The following are equivalent facts:
\begin{itemize}
\item[\em (i)] \quad $\gl w(\hx)-\II\varphi(\hat x)+H(\hat x,D\varphi(\hat x))\geq 0$;
\item[\em (ii)] \quad for every $\delta\in (0,1)$, 
\[
\gl w(\hx)-\II_\gd^1\varphi(\hat x)
-\II_\gd^2 (w,D\varphi(\hat x),\hat x)
+H(\hat x,D\varphi(\hat x))
\geq 0,
\] 
where 
\begin{align}
& \II_\gd^1\varphi(x):=\int_{B_\gd} \big(\varphi(x+j(x,z))-\varphi(x)-\du{j(x,z),D\varphi(x)}\big)\,\nu(dz), \label{def11}
\\&\II_\gd^2( w, p,x):=\int_{\R^d\stm B_\gd} \big(w(x+j(x,z))-w(x)-\1_B(z)\du{j(x,z),p}\big)\,\nu(dz). \label{def12}
\end{align}
\item[\em(iii)] \quad $
\gl w(\hx)-I(w,D\varphi(\hx),\hat x)+H(\hat x,D\varphi(\hat x))\geq 0,
$
where 
\[
{\II (w,p,x)}:=\int_{\R^d}\big(w(x+j(x,z))-w(x)-\1_B(z)\lan p,j(x,z)\ran\big)\,\nu(dz);
\]
\end{itemize}
Furthermore, if either one of the above facts holds, then 
\beq\label{eq inequality I}
\II(u,D\varphi(\hx),\hx)\geq \II_\gd^1\varphi(\hx)+\II_\gd^2(u,D\varphi(\hx),\hx)\geq 
\II\varphi(\hx)
\eeq
and each term in this inequality is finite. 
\end{theorem}

\begin{proof}
The fact that item {\em (i)} holds amounts to saying that the upper semicontinuous function  
$u:=-w$ satisfies the subsolution test at $\hat x$ for the following nonlocal HJ equation
\[
\gl u-\II u-H (x,-Du)=0  \ \ \IN \T^d. 
\]
The equivalence of (i), (ii) and (iii) follows by 
by applying Theorem \ref{thm sub equivalent} to $u:=-w$ with $\tilde H(x,p):=-H(x,-p)$ in place of $H$. 
\end{proof}

We record here for further use the following application of Theorem \ref{thm sub equivalent}.

\begin{proposition}\label{prop pointwise subsol}
Let $\gl\geq 0$ and let $u$ be a bounded and semiconvex function on $\R^d$. Let us denote by $\Sigma_u$ the negligible subset of $\R^d$ made up by points where $u$ is not twice differentiable. 
Then $u$ is a viscosity subsolution to \eqref{eq discounted} if and only if 
\beq\label{eq pointwise subsolution}
\gl u(x)-\II u(x)+H (x,Du(x))\leq 0\quad\FORALL x\in\R^d\setminus \Sigma_u.
\eeq
\end{proposition}

\bproof
Let us first prove the {\em only if} part. Pick a point $\hx\in\R^d\setminus\Sigma_u$. Being $u$ twice differentiable at $\hx$, according to \cite[Section 2]{users} there exists a function $\varphi\in\C^2(\R^d)$ such that $u-\varphi$ has a local maximum at $\hx$. Up to suitably modifying $\varphi$, we can furthermore assume that $\hx$ is a global maximum point of $u-\varphi$ and that $\varphi$ is bounded, being $u$ bounded. Since $Du(\hx)=D\varphi(\hx)$, we derive that  
$\II (u,D\varphi(\hx),\hx)=\II (u,D u(\hx),\hx)=\II u(\hx)$. Being $u$ a subsolution to  \eqref{eq discounted}, we infer, according to Theorem \ref{thm sub equivalent}, 
\[
\gl u(\hx)-\II u(\hx)+H (\hx,Du(\hx))
=
\gl u(\hx)-\II (u,D\varphi(\hx),\hx)+H (\hx,D\varphi(\hx))
\leq 0,
\]
as it was to be shown. 

Let us now prove the {\em if} part. 
Let $\varphi\in \C^2_{{b}}(\R^d)$ be such that 
$u-\varphi$ has a strict maximum at some point $\hx\in\Rd$. By semiconcavity, $u$ is differentiable at $\hx$ and $Du(\hx)=D\varphi(\hx)$, hence $\II (u,D\varphi(\hx),\hx)=\II (u,D u(\hx),\hx)=\II u(\hx)$. 
Furthermore, by Proposition \ref{max1}, there exists a sequence 
$(\varphi_k,x_k)\in \C^2(\R^d)\tim \R^d$ such that  $x_k\in\Rd\setminus\Sigma_u$ and $u-\varphi_k$ attains a global maximum at $x_k$, for each $k\in\N$,  and $x_k\to\hx$, $\varphi_k \to \varphi \ \IN \C^2(\R^d)$ as $k\to +\infty$. 
We derive that $D\varphi_k (x_k)=D u(x_k)$ and 
$\II\varphi_k(x_k)=\II (u,D\varphi_k(x_k),x_k)=\II (u,D u(x_k),x_k)=\II u(x_k)$, hence, by making use of \eqref{eq pointwise subsolution}, we get
\[
\gl u(x_k)-\II \varphi_k(x_k)+H (x_k,D\varphi_k(x_k))
=
\gl u(x_k)-\II u(x_k)+H (x_k,D u(x_k))
\leq 0.
\]
By sending $k\to +\infty$, from the fact that $\II\varphi_k\to\II\varphi$ in $\C(\Rd)$ we derive
\[
\gl u(\hx)-\II \varphi(\hx)+H (\hx,Du(\hx))\leq 0.
\]
The proof is complete.
\eproof

\subsection{Existence and uniqueness of discounted solutions}
We proceed with the analysis of equation \eqref{eq general discounted} by focusing on the periodic case $M=\M$ with $\lambda>0$, i.e., 
\begin{equation}\label{eq0}
\gl u-\II u+H (x,Du)=0 \ \ \ \IN \T^d \tag{HJ$^\lambda_\M$}.
\end{equation}
Throughout this section, we will assume conditions {\bf (N),\,(J1),\,(J2)} and {\bf (H1)} with $M:=\M$. 
%
%

%
%

We proceed by showing that, for $\lambda>0$, equation \eqref{eq0} satisfies a comparison principle.

\begin{proposition}\label{prop comparison}  Assume {\bf (N),\,(J1),\,(J2),\,(H1)}  
and $\lambda>0$. Let $v\in\Bd(\T^d)\cap\USC(\T^d)$ and $w\in\Bd(\T^d)\cap\LSC(\T^d)$ be sub and super solutions to \eqref{eq0}, 
respectively. 
Then $v\leq w$ on $\T^d$. 
\end{proposition}

\bproof Suppose, by contradiction, that $\max\limits_\M(v-w)>0$. Let $k\in\N$ and consider the function 
\[
\Phi_k(x,y)=v(x)-w(y)-k|x-y|^2 \ \ \ON \R^{2d},
\]
where $v,w$ are considered as periodic functions on $\R^d$. Let $(x_k,y_k)$ be a 
maximum point of $\Phi_k$. We may assume, by passing to a subsequence, that 
\[
\lim_k k|x_k-y_k|^2=0 \quad \AND\quad \lim_k x_k=x_0
\]
for some $x_0\in\R^{d}$. It follows that $\lim_k y_k=x_0$. 
The standard observation is that 
\[
\lim_k v(x_k)=v(x_0), \quad \lim_k w(y_k)=w(x_0), \quad\AND\quad (v-w)(x_0)=\max(v-w)>0.
\]
Since  
\[
v(x_k)-w(y_k)-k|x_k-y_k|^2\geq v(\cdot)-w(y_k)-k|\cdot-y_k|^2 \ \  \IN \R^d,
\] 
in view of Theorem \ref{thm sub equivalent}
we obtain
\beq\label{eq5}
\gl v(x_k)-\II (v,2k(x_k-y_k),x_k)+H (x_k,2k(x_k-y_k))\leq 0,
\eeq
where 
\[
{\II (v,p,x)}:=\int_{\R^d}\big(v(x+j(x,z))-v(x)-\1_B(z)\lan p,j(x,z)\ran\big)\,\nu(dz);
\]
Similarly, 
since  
\[
v(x_k)-w(y_k)-k|x_k-y_k|^2\geq v(x_k)-w(\cdot)-k|x_k-\cdot|^2 \ \  \IN \R^d,
\] 
in view of Theorem \ref{thm super equivalent} we obtain 
\beq\label{eq6}
\gl w(y_k)-\II (w,2k(x_k-y_k),y_k)+H (y_k,k(x_k-y_k))\geq 0.
\eeq
By subtracting \eqref{eq6} from \eqref{eq5} and by using assumption (H1), we get 
\begin{flalign}\label{eq6a}
\gl (v(x_k)&-w(y_k))
\\
&\leq 
\omega_H\left(|x_k-y_k| +2k|x_k-y_k|^2\right)+
\II (v,2k(x_k-y_k),x_k)
-
\II (w,2k(x_k-y_k),y_k). \nonumber
\end{flalign}
We want estimate the right-hand side of \eqref{eq6a}. To this aim, we start with remarking that, by definition of $(x_k,y_k)$, for all $z\in\R^d$ we have 
\begin{flalign*}
v(x_k)-w(y_k)&-k|x_k-y_k|^2\\
&\geq 
v(x_k+j(x_k,z))-w(y_k+j(y_k,z))-k|x_k+j(x_k,z)-y_k-j(y_k,z)|^2, 
\end{flalign*}
yielding
\begin{flalign*}
\Big(
v(x_k+j(x_k,z))
-
v(x_k)
\Big)
&-
\Big(
w(y_k+j(y_k,z))-w(y_k)
\Big)
\\
&\leq
k|j(x_k,z)-j(y_k,z)|^2
+
2\langle k(x_k-y_k),j(x_k,z)-j(y_k,z)\rangle.
\end{flalign*}
For every $R>1$, we derive
\begin{flalign*}
\II (v,&2k(x_k-y_k),x_k)
-
\II (w,2k(x_k-y_k),y_k)
\leq
2(\|v\|_\infty+\|w\|_\infty) \nu\left(\R^d\setminus B_R\right)\\ 
&+
k \int_{B_R} |j(x_k,z)-j(y_k,z)|^2\,\nu(dz) 
+
2k |x_k-y_k| \int_{B_R\setminus B}  |j(x_k,z)-j(y_k,z)|\,\nu(dz)
\\
&{\leq}		
2(\|v\|_\infty+\|w\|_\infty)\nu\left(\R^d\setminus B_R\right)
+
(R^2C_jC_\nu+2C_1) \,
k|x_k-y_k|^2,
\end{flalign*}
where, for the last inequality, we have used the assumptions {\bf (N), (J1), (J2)}. 
By sending first $k\to +\infty$ and then $R\to +\infty$ in the previous inequality, we derive, in view of assumption {\bf (N)}, that 
\[
\limsup_{k\to +\infty}
\Big( 
\II (v,2k(x_k-y_k),x_k)
-
\II (w,2k(x_k-y_k),y_k)
\Big)
\leq 
0.
\]
By sending $k\to+\infty$ in \eqref{eq6a}, we finally  conclude that 
$\gl(v-w)(x_0)\leq 0$, leading to a contradiction. 
\eproof

Next, we prove the existence of a (unique) continuous solution to \eqref{eq0} when $\lambda>0$. 

\begin{theorem} \label{lem.ma14} 
 Assume {\bf (N),\,(J1),\,(J2),\,(H1)}  
and $\lambda>0$.
Then there is a unique solution $u\in \C(\T^d)$ to 
\eqref{eq0}. 
\end{theorem}

\bproof 
We are going to show the existence of a continuous solution to \eqref{eq0}, which is also unique due to Proposition  \ref{prop comparison}. To this aim, we first note that any constant function $u$ satisfies $\II u(x)=0$ for all $x\in\T^d$. Hence, 
for a suitable large constant $C>0$, the functions $v_0(\cdot)= -C$ and $w_0(\cdot)=C$ are, respectively, a sub and a supersolution to \eqref{eq0}. We define a function $u:\M\to\R$ by setting  
\beq
u(x):=\sup\left\{ v(x)\,:\, \hbox{$v\in\USC(\M)$ is a subsolution to \eqref{eq0},\ $v\geq-C$ on $\M$}\,\right\}
\eeq
for every $x\in\M$. 
By the Comparison Principle stated in Proposition \ref{prop comparison}, we know that all subsolutions $v\in\USC(\M)$ of \eqref{eq0} satisfy $v\leq C$ on $\M$, hence the same holds for $u$. 
Let us denote by $u^*,\, u_*$ the upper and lower semicontinuous envelope of $u$ defined according to \eqref{def semicontinuous envelope}.
By Proposition \ref{prop sup inf}-(a) we know that $u^*$ is a (upper semicontinuous) subsolution of \eqref{eq0}. Let us show that $u_*$ is a (lower semicontinuous) supersolution of \eqref{eq0}. Let us assume by contradiction that there exists $\psi\in \C^2(\M)$ such that $u_*-\psi$ attains a global strict  minimum at $\hat x$ with $(u_*-\psi)(\hat x)=0$ and 
\[
\lambda \psi(\hat x)-\II \psi(\hat x)+H (\hat x, D\psi(\hat x))<0.
\]
By the fact that $\psi$ is of class $C^2$ and in view of Lemma \ref{lemma I well defined}, we infer that there exists $\rho_0>0$ such that 
\beq\label{eq local subsol}
\lambda \psi(x)-\II \psi(x)+H (x, D\psi(x))<0
\qquad
\hbox{for all $x\in B_{2\rho_0}(\hat x)$}.
\eeq 
Let us pick a bump function $\chi\in C^2(\M)$ such that $\chi\geq 0$ on $\M$, $\chi(0)=1$ and $\supp(\chi)\subset B_1$. For every $\ep>0$, let us set
\[
\psi_\ep(x):=\psi(x)+\ep^2\chi\left(\frac{x-\hat x}{\ep}\right),
\quad
u_\eps(x):=\max\{u(x),\psi_\ep(x)\}
\qquad
\hbox{for all $x\in\M$.} 
\]
For every $\eps>0$ small enough, $u_\eps=u$ in $\M\setminus B_{\rho_0}(\hat x)$ and the inequality \eqref{eq local subsol} holds with $\psi_\eps$ in place of $\psi$. By Proposition \ref{prop sup inf}-(a) we infer that $u_\eps^*$ is a subsolution of \eqref{eq0} satisfying $u^*_\eps\geq -C$ on $\M$. This contradicts the definition of $u$ since $u^*_\eps(\hat x)\geq u_\eps(\hat x)=\psi_\eps(\hat x)>\psi(\hat x)=u(\hat x)$. 
\eproof

\subsection{The critical equation} Let $H$ be a Hamiltonian satisfying {\bf (H1)} and consider the following nonlocal HJ equation 
\beq\label{eq eikonal HJ a}
-\II u+H(x,Du)=a\quad \IN\ \M,\tag{HJ$_{a}$}
\eeq
where $a\in\R$ and  {\bf (N),\,(J1),\,(J2)} are in force. We first show that equation \eqref{eq eikonal HJ a} admits subsolutions when $a$ is large enough. 
 
\begin{lemma}\label{lemma good definition}
Let $u\in\C^2(\M)$. Then $u$ is a subsolution to \eqref{eq eikonal HJ a} with 
\[
a:=\|\II u\|_\infty+\|H(x,Du(x))\|_\infty.
\]  
\end{lemma}

\bproof
First note that $\II u\in\C(\M)$, in view of Lemma \ref{lemma I well defined}, yielding that the $a$ defined above is actually a real number. By definition of $a$, we have 
\[
-\II u(x)+H(x,Du(x))\leq a\quad\FORALL\ x\in\M,
\]
which implies that $u$ is a viscosity subsolution of  \eqref{eq eikonal HJ a} by Proposition \ref{prop pointwise subsol}.
\eproof

Let us  set
\beq\label{eq critical value}
c:=\inf\left\{ a\in\R\,:\, \hbox{equation \eqref{eq eikonal HJ a} admits subsolutions} \right\}. 
\eeq
By Lemma \ref{lemma good definition}, the set appearing at the right-hand side of the above expression is nonempty. The next result shows that such a constant $c$ is finite. 
\begin{lemma}\label{lemma critical value}
We have\ \  $c\geq \min\limits_{x\in\M} H(x,0)$. 
\end{lemma}

\bproof
Fix $a>c$ and take a subsolution $u\in\C(\M)$ of \eqref{eq eikonal HJ a}. Let $x_0\in\M$ be a maximum point of $u$ on $\M$. Then any constant function $\varphi$ is a supertangent to $u$ at $x_0$. Being $u$ a subsolution to \eqref{eq eikonal HJ a}, we get 
\[
a\geq -\II \varphi(x_0)+H(x_0,D\varphi(x_0))=H(x_0,0)\geq \min_{x\in\M} H(x,0). 
\]
By the arbitrariness of the choice of $a>c$, we get the assertion.
\eproof

We end this section by showing the following fact. 

\begin{proposition}\label{prop LPV}
Let $u_1, u_2\in\C(\M)$ be a sub and a supersolution of equation \eqref{eq eikonal HJ a} with 
$a:=a_1$ and $a:=a_2$, respectively. Then $a_2\leq a_1$. In particular, if \eqref{eq eikonal HJ a} admits solutions, then $a=c$. 
\end{proposition}

\bproof
The argument is standard, see \cite{LPV}. Let us assume, by contradiction, that $a_1<a_2$ and choose $\eps>0$ small enough so that $a_1+\eps<a_2-\eps$. 
Up to adding to $u_1$ a suitably large constant, we can assume that $u_1>u_2$ in 
$\M$. Choose $\gl>0$ small enough so that $\gl \|u_i\|_\infty<\eps$ for each $i\in\{1,2\}$. Then 
\[
\gl u_1-\II u+H(x, Du_1)
<
a_1+\eps
<a_2-\eps
<
\gl u_2-\II u_2+H(x, Du_2)
\quad
\IN\ \M
\]
in the viscosity sense. By Proposition \ref{prop comparison}, we derive that $u_1\leq u_2$ in $\M$, yielding a contradiction. 
\eproof

The constant $c$ given by \eqref{eq critical value} will be referred to as the {\em critical constant} associated with $H$, consistent with the case where $\II\equiv 0$. It is characterized as the unique real constant $a$ for which 
the equation \eqref{eq eikonal HJ a} may admit solutions, in view of Proposition \ref{prop LPV}.

\section{Duality techniques for nonlocal convex HJ equations}\label{sec.duality}

Let $Q$ be a nonempty compact subset of $\R^d$ and set $K:=\M\times Q$.
For any given $\phi\in \C(K)$ we set 
\beq\label{def H phi}
H_\phi(x,p):=\max_{\xi\in Q} [\langle p,\xi\rangle-\phi(x,\xi)] \ \ \FORALL (x,p)\in\TM.\tag{H$_\phi$}
\eeq
As a maximum of linear functions in $p$, the Hamiltonian $H_\phi$ clearly satisfies {\bf (H2)}. It also satisfies {\bf (H1)}. Indeed, if 
$\go_\phi$ is a continuity modulus of $\phi$ in $K$, we have 
\[
|H_\phi(x,p)-H_\phi(y,p)|\leq \go_\phi(|x-y|)
\quad
\FORALL x,y\in\M \ \AND \ p\in\R^d. 
\]
For every fixed $\lambda \geq  0$, we will denote by ${\cF}(\gl)$ the set 
of  $(\phi,u)\in \C(K)\tim \C(\T^d)$ such that $u$ is a viscosity subsolution of
\beq\label{eq HJ phi}
\gl u-\II u+H_\phi(x,Du)= 0\ \ \IN \T^d.
\tag{HJ$^\gl_\phi$}
\eeq
The nonlocal differential operator $\II$ appearing above is of the form \eqref{def I} with $M:=\M$ and where conditions {\bf (N),\,(J1),\,(J2)} are in force.  
For $z\in\T^d$, we set
\[
{\cG}(z,\gl):=\{\phi-\gl u(z)\mid (\phi,u)\in {\cF}(\gl)\}\subseteq \C(K). 
\]
Note that, when $\lambda=0$, the set ${\cG}(z,\gl)$ is actually independent of $z$ and it reduces to the projection of ${\cF}(\gl)$ in the first component, i.e.
\[
{\cG(0)}:=
{\cG}(z,0)=\{\phi\mid (\phi,u)\in {\cF}(0)\ \ \hbox{for some $u\in \C(\M)$}\}.
\]

We start with the following easy remarks, which imply, in particular, that  
${\cG}(z,\gl)$ is nonempty, for every $\lambda\geq 0$ and $z\in\M$. 
Hereafter, we will write $\0_\M$ for the function $u(x)\equiv 0$ on $\M$, $\1_\M$ for the function $u(x)\equiv 1$ on $\M$, 
and $\0_K$ the function $\phi\equiv 0$ on $K$.

\begin{lemma}\label{lemma 0 in cG} Let $\lambda \geq  0$ and $z\in\M$. 
Any $\phi \in \C(K)$ satisfying $\phi\geq 0$ on $K$ belongs to ${\cG}(z,\gl)$. 
In particular, the functions $\0_K$ and $\1_K$ belong to $\cG(z,\gl)$ and to the interior of 
$\cG(z,\gl)$, respectively. 
\end{lemma}

\bproof  
Fix $\phi \in \C(K)$ satisfying $\phi\geq 0$ on $K$, and 
note that, for $u={\0_\M}$, 
\[
\gl u-\II u+H_\phi(x,Du)=\max_{\xi\in Q}(-\phi(x,\xi))\leq 0.
\]
Hence $(\phi,{\0_\M})\in {\cF}(\gl)$ and $\phi-\gl {\0_\M}(z)=\phi\in{\cG}(z,\gl)$. In particular, this implies that $\0_K\in \cG(z,\gl)$.  

Let us show that  $\1$ is in the interior of $\cG(z,\gl)$. To this aim, consider the open neighborhood $V$ of $\1$ in $\C(K)$ defined as follows: 
\[
V:=\left\{\phi\in\C(K)\,\mid\, \|\phi-\1\|_{\infty}<1\,\right\}.
\]
Any $\phi\in V$ satisfies $\phi>0$ on $K$, hence from the above we infer 
$
(\phi,{\0_\M})\in {\cF}(\gl)$ 
\ for all $\phi\in V$, 
thus showing that $V$ is contained in ${\cG}(z, \gl)$. The proof is complete. 
\eproof 

We proceed by establishing the following important fact. 

\begin{proposition} \label{prop cG convex} Let $\gl\geq 0$ and $z\in\T^d$. Then ${\cF}(\gl)$ and ${\cG}(z,\gl)$ are convex subsets 
of $\C(K)\tim \C(\T^d)$ and $\C(K)$, respectively. Furthermore, $\cF(\gl)$ is closed.
\end{proposition}

\bproof 
Let $(\phi,v), (\psi,w)\in{\cF}(\gl)$ be arbitrarily fixed. Pick $n\in\N$ and let 
$v_n:=v^{\ep_n}$ and $w_n:=w^{\ep_n}$ be the sup-convolutions 
of $v$ and $w$, respectively, with $\eps_n>0$. According to Proposition \ref{prop sup convolution}, we can choose $\eps_n>0$ small enough  so that
$\|v-v_n\|_\infty<1/n$,  $\|w-w_n\|_\infty<1/n$
and 
\begin{flalign*}
&\gl v_n -\II v_n +H_\phi(x,Dv_n )\leq 1/n \ \ \ \IN \T^d,\\
&\gl w_n-\II w_n +H_\psi(x,Dw_n )\leq 1/n \ \ \ \IN \T^d,
\end{flalign*}
in the viscosity sense. 

Let $f\in C^2(\T^d)$ and $\hat x\in\T^d$ be a maximum point of $v_n +w_n  -f$. 
We may assume that $(v_n +w_n -f)(\hat x)=0$.
Note that $w_n $ is semiconvex, hence there is a function $g\in \C_{\mathrm{b}}^2(\R^d)$ such that 
$(w_n -g)(\hat x)=\min_{\R^d}(w_n -g)=0$ and $w_n \geq g$ in $\R^d$. Then
\[
(v_n -f+g)(\hat x)=(v_n +w_n -f)(\hat x) \geq  (v_n +w_n -f)(x)
\geq (v_n -f+g)(x)
\quad
\FORALL x\in\R^d.
\]
By the subsolution property of $v_n $, we derive, according to Theorem \ref{thm sub equivalent},
\beq\label{eq1 approximate subsol}
\gl v_n (\hat x) -\II(v_n,D(f-g)(\hx),\hat x)+H_\phi(\hat x, D(f-g)(\hat x))
\leq 1/n,
\eeq
where 
\[
{\II (u,p,x)}:=\int_{\R^d}\big(u(x+j(x,z))-u(x)-\1_B(z)\lan p,j(x,z)\ran\big)\,\nu(dz).
\]
Similarly, we may choose $h\in \C_{\mathrm{b}}^2(\R^d)$ such that 
$(v_n -h)(\hat x)=\min_{\R^d}(v_n  -h)=0$. Observe that 
\[
(w_n -f+h)(\hat x)=(v_n +w_n -f)(\hat x)\geq (v_n +w_n -f)(x)\geq (w_n -f+h)(x)
\ \ \FOR x\in\R^d,
\] 
hence, by the subsolution property of $w_n $, we derive, according to Theorem \ref{thm sub equivalent}, 
\beq\label{eq2 approximate subsol}
\gl w_n (\hat x)
-
\II(w_n,D(f-h)(\hx),\hat x)
+H_\psi(\hat x, D(f-h)(\hat x))
\leq 
1/n. 
\eeq
Note that 
\[
f(x)\geq (v_n +w_n )\geq (h+g)(x) \ \ \FORALL x\in\R^d \quad\AND \quad f(\hat x)=(v_n +w_n )(\hat x)=(h+g)(\hat x),
\]
i.e., $f-h-g$ has a minimum at $\hat x$, which implies that 
\[
D(f-h-g)(\hat x)=0.
\]
By definition of $H_\phi$,\,$H_\psi$, it is easily seen that, for all $x\in\M$ and $p_1,\,p_2\in\R^d$,  
\[
H_\phi(x,p_1)
+
H_\psi(x,p_2)
\leq
\max_{\xi\in Q} \left[ \langle p_1+p_2,\xi\rangle-\big(\phi(x,\xi)+\psi(x,\xi)\big) \right]
=
H_{\phi+\psi}(x,p_1+p_2).
\]
By summing up \eqref{eq1 approximate subsol} and \eqref{eq2 approximate subsol}, we obtain, in view of the above,
\[
\gl(v_n +w_n )(\hat x)
-
\II(v_n,D(f-g)(\hx),\hat x)
-
\II(w_n,D(f-h)(\hx),\hat x)
+
H_{\phi+\psi} (\hx,Df(\hx))
\leq 
2/n.
\]
Now, in view of the inequalities \eqref{eq1 approximate subsol}, \eqref{eq2 approximate subsol} and \eqref{eq inequality I},  we derive that 
\[
\II(v_n,D(f-g)\hx),\hat x)\in\R
\quad
\AND
\quad 
\II(w_n,D(f-h)(\hx),\hat x)\in\R,
\]
meaning that the functions 
\begin{eqnarray*}
&z\mapsto v_n(\hx+j(\hx,z))-v_n(\hx)- \1_B(z)\lan D(f-g)(\hx),j(\hx,z)\ran\\
&z\mapsto w_n(\hx+j(\hx,z))-w_n(\hx)- \1_B(z)\lan D(f-h)(\hx),j(\hx,z)\ran
\end{eqnarray*}
belong to $L^1(\R^d,\nu)$.
This implies that 
\[
\II(v_n,D(f-g)(\hx),\hat x)
+
\II(w_n,D(f-h)(\hx),\hat x)
=
\II(v_n+w_n,Df(\hx),\hat x),
\]
yielding
\[
\gl(v_n +w_n )(\hat x)
-
\II(v_n+w_n,D f(\hx),\hat x)
+
H_{\phi+\psi}(\hat x,Df(\hat x))
\leq 
2/n.
\]
This shows that $v_n +w_n $ is a subsolution of 
\[
\gl u-\II u+H_{\phi+\psi}(x,Du)=2/n \ \ \IN \T^d.
\]
By sending $n\to +\infty$, we conclude, by the stability of the notion of viscosity subsolution, that $v+w$ is a subsolution to 
\[
\gl u-\II u+H_{\phi+\psi}(x,Du)=0 \ \ \IN \T^d,
\]
i.e., $(v+w,\phi+\psi)\in{\cF}(\gl)$. 
It is clear that $(tu,t\phi)\in{\cF}(\gl)$ for any $t\geq  0$.
This implies that ${\cF}(\gl)$, and consequently ${\cG}(z,\gl)$, is a convex cone. 
The fact that ${\cF}(\gl)$ is closed is again a consequence of the stability property 
of viscosity subsolutions.
\eproof

Next, we show that $\overline{\cG(z,\gl)}$ is a proper subset of $\C(K)$ for every $z\in\M$ and $\gl\geq 0$. We begin by remarking what follows. 

\begin{lemma}\label{lemma supersolution HJ phi} Let $\gl\geq  0$ and $z\in\M$. Let $\phi\in  \C(K)$ and let $u\in \C(\T^d)$ be a supersolution of
\beq\label{eq HJ supersolution}
\gl u-\II u +H_\phi(x,Du)=\eps\quad\IN\ \M,
\eeq
where $\ep>0$ is a constant. 
Then we have $\phi-\gl u(z)\not\in \cG(z,\gl)$.
\end{lemma} 

\bproof Fix $z\in\T^d$. By contradition, we suppose that 
$\phi- \gl u(z)\in\cG(z,\gl)$. Then, there exists $(\psi,v)\in \cF(\gl)$ such that 
$\psi-\gl v(z)=\phi-\gl u(z)$. 
We get   
\[\bald
0&\geq \gl v-\II v+H_\psi(x,Dv)
=\gl v-\II v+H_{\phi-\gl(u(z)-v(z))}(x,D v)
\\&
=\gl(v+u(z)-v(z))-\II v+H_{\phi}(x, Dv)
=\gl \hat v-\II \hat v+H_{\phi}(x,D\hat v),
\eald
\]
where we have set $\hat v:= v+u(z)-v(z)$.
On the other hand, if $\gl>0$, 
we have that $u-\eps/\gl$ is a supersolution of \eqref{eq HJ phi} and, by comparison, we find that 
\[
v+u(z)-v(z)\leq u-\ep/\gl\quad\IN\T^d, 
\]
which cannot be, as it is easily seen by evaluating the above inequality at the point $x:=z$. 
If $\gl=0$, we have $-\II\hat v+H_{\phi}(x,D\hat v)\leq 0$ and $-\II u+H_\phi(x,Du)\geq \ep$ in $\T^d$, which contradicts Proposition \ref{prop LPV}. 
\eproof

\smallskip

\begin{proposition}\label{prop nontrivial cone} Let $\gl\geq 0$ and $z\in\T^d$. 
Then $\0_K\in \pl \cG(z,\gl)$. In particular, $\overline{\cG(z,\gl)}$ is strictly contained in $\C(K)$.
\end{proposition}

\bproof Note that $\0_{\T^d}$ is a solution of  \eqref{eq HJ phi}, with $\phi=\0_K$. 
Hence, Lemma \ref{lemma supersolution HJ phi} shows that $\0_K\in \pl\cG(z,\gl)$. The last assertion follows as a consequence of the Hahn-Banach Theorem applied to the sets $\Dr{int}(\cG(z,\gl))$ and $\{\0_K\}$.   
\eproof 

When $\gl>0$, we furthermore have the following fact.

\begin{proposition} \label{prop G(lambda,z) closed} Let $\lambda> 0$ and $z\in\M$. 
Then ${\cG}(z,\gl)$ is a closed and proper subset of $\C(K)$.
\end{proposition}

\bproof
Pick $\phi\in\C(K)$ such that $\phi\not\in \cG(z,\gl)$, which exists in view of Proposition \ref{prop nontrivial cone}. This means that any $u\in\C(\M)$ such that $(\phi,u)\in\cF(\gl)$ satisfies $u(z)\not=0$, otherwise stated, that the set 
\beq\label{eq convexity rmk}
S:=\{\gl u(z)\,\mid\,u\in\C(\M)\ \hbox{such that}\ (\phi,u)\in\cF(\gl)\,\}
\eeq
does not contain $0$.  
%
Let us denote by $u_\gl$ the unique continuous solution to \eqref{eq HJ phi}. It is easily seen that $(\phi,u_\gl-k)\in\cF(\gl)$ for every $k\geq 0$. Furthermore, any $u\in\C(\M)$ such that $(\phi,u)\in\cF(\gl)$ satisfies $u\leq u_\gl$, 
in view of the Comparison Principle stated in Proposition \ref{prop comparison}. 
We conclude that  
\beq\label{eq S}
S= (-\infty,\gl u_\gl(z)]=(-\infty,-\rho]\quad\FOR \hbox{some}\ \rho>0.
\eeq
Let us fix $r<\rho$ and consider the open neighborhood $V$ of $\phi$ in $\C(K)$ defined as follows: 
\[
V:=\left\{\psi\in\C(K)\,\mid\, \|\psi-\phi\|_{\infty}<r\,\right\}.
\]
We claim that $V\cap \cG(z,\gl)=\emptyset$. Indeed, let us assume by contradiction that there exists $\psi \in V\cap  \cG(z,\gl)$. Then there exists $u\in\C(\M)$ such that $(\psi,u)\in\cF(\gl)$ and $u(z)=0$. Now we have 
\[
\gl u -\II u+H_\phi(x,Du)
\leq
\gl u- \II u+H_\psi(x,Du)+\|\phi-\psi\|_{\infty}
<
r
\quad
\IN\ \M,
\]
i.e., $(\phi, u-r/\gl)\in\F(\gl)$. By definition of $S$, we get that $-r=\gl (u(z)-r/\gl)\in S$, yielding a contradiction with \eqref{eq S} by the choice of $r$. 
\eproof

Next, we prove two auxiliary lemmas that we shall need in the proof of the last two theorems of this section.

\begin{lemma} \label{lem.ma24} Let $\lambda>  0$, $z\in\M$ and $\phi\in\C(K)$. 
 Let $u_\lambda\in\C(\T^d)$ be the unique solution to \eqref{eq HJ phi}.
Then 
\[
\phi-\gl u_\gl(z)-\ep\not \in {\cG}(z,\gl)\quad\FORALL \eps>0.
\]
In particular, $\phi-\gl u_\gl(z)\in \pl {\cG}(z,\gl)$.
\end{lemma}

\bproof From the fact that $u_\gl$ is a solution of \eqref{eq HJ phi} and $H_{\phi-\eps}=H_\phi+\eps$, it is easily seen that $u_\gl$ is a supersolution of \eqref{eq HJ supersolution}, with $\phi-\e$ in place of $\phi$. This proves the first assertion in view of Lemma \ref{lemma supersolution HJ phi}. The second one follows from the fact that 
$\phi-\gl u_\gl(z)\in{\cG}(z,\gl)$ since $(\phi,u_\gl)\in\cF(\gl)$.   
\eproof

\begin{lemma}\label{lemma phi+c boundary point}
Let $\phi\in \C(K)$ and denote by  $c_\phi$ the critical constant associated with $H_\phi$ through \eqref{eq critical value}. Then 
\[
\phi+c_\phi+\eps\in \cG(0)
\quad
\AND
\quad
\phi+c_\phi-\eps\not\in \cG(0)
\qquad
\FORALL\ \eps>0.
\]
In particular, $\phi+c_\phi\in \partial\cG(0)$
\end{lemma}

\bproof
By definition of $c_\phi$, the equation
\[
-\II u+H_\phi(x,Du)=a\quad\IN\ \M
\]
admits viscosity solutions for every $a>c_\phi$, and does not admit viscosity solutions for any $a<c_\phi$. This readily proves the assertion for $H_{\phi+a}=H_\phi-a$. 
\eproof

For every fixed $\gl\geq0$ and $z\in\M$, we denote by ${\cGprime}(z,\lambda)$ the set of probability measures $\parts(K)$  on $K$ that belong to the dual cone of ${\cG}(z,\lambda)$, that is 
\[
{\cGprime}(z,\lambda):=\left\{ \mu\in\parts(K)\,\mid\, \langle \mu,  f\rangle \geq 0\quad \FORALL f\in {\cG}(z,\gl) \right\}.
\]
The following holds. 

\begin{proposition}\label{prop nonempty dual cone}
Let $\gl\geq 0$ and $z\in\M$. 
Then ${\cGprime}(z,\lambda)$ is nonempty. 
\end{proposition}

\bproof
By Lemma \ref{lemma 0 in cG} and Propositions \ref{prop cG convex} and \ref{prop nontrivial cone}, we 
know that $\overline{\cG(z,\gl)}$ is a proper closed and convex subset of $\C(K)$ with nonempty interior, in particular 
there exists a function $f_0$ in $\C(K)$ such that $f_0\not\in \overline{{\cG}(z,\gl)}$. By Hahn-Banach Theorem, there exists $\mu\in\C(K)^*$, $\mu\not=0$, such that 
\[
\lan \mu,f_0\ran
<
\lan \mu,f\ran
\qquad
\FORALL f\in  {\cG}(z,\gl).
\]
Being ${\cG}(z,\gl)$ a positive cone, this is only possible if 
\[
\lan \mu,f\ran
\geq 
0
\qquad
\FORALL f\in  {\cG}(z,\gl).
\]
By Lemma \ref{lemma 0 in cG}, any nonnegative $\chi\in\C(\TM)$ is in ${\cG}(z,\gl)$. Hence,
\[
\lan\mu,\chi\ran\geq 0 \quad\FORALL \chi\in C_+(K),
\]
where $C_+(K):=\{g\in\C(K)\mid g\geq 0 \ \ON\  K\}$. 
By Riesz Theorem, $\mu$ can be identified with a positive Radon measure with compact support. 
By normalization, we may assume that $\mu$ is a probability measure on $K$, finally showing that $\mu\in{\cGprime}(z,\gl)$.
%
\eproof

We are now in position to prove the following theorems, that will be crucial for our forthcoming analysis. We start with the case $\gl>0$. 

\begin{theorem}\label{thm discounted Mather measures} Let $\lambda>0$, $z\in\T^d$ and $\phi\in\C(K)$. Let $u_\lambda\in\C(\T^d)$ be the 
unique solution to \eqref{eq HJ phi}. Then 
there exists a measure $\mu_\lambda\in{\cGprime}(z,\lambda)$
such that
\begin{equation}\label{claim discounted minimization}
\gl u_\lambda(z)=\int_K\phi(x,\xi)\,\mu_\lambda(dxd\xi)=\min_{\mu\in{\cGprime}(z,\lambda)}\int_K \phi(x,\xi)\,\mu(dxd\xi).
\end{equation}
\end{theorem}


\bproof By definition of $u_\lambda$ and ${\cG'}(z,\gl)$, it is clear that $\phi-\gl u_\gl(z)\in{\cG}(z,\gl)$ and 
$\langle \mu, \phi-\gl u_\gl(z)\rangle \geq 0$ for all $\mu\in {\cG'}(z,\gl)$, namely
\[
\int_K \phi(x,\xi)\,\mu(dxd\xi)\geq \gl u_\gl(z)\qquad\FORALL \mu\in {\cG'}(z,\gl).
\]
Let us proceed to show the existence of a measure $\mu_\gl\in {\cG'}(z,\gl)$ that 
realizes the equality in the above inequality. 
By Lemma \ref{lemma 0 in cG} and  Propositions \ref{prop cG convex} and \ref{prop G(lambda,z) closed}, we know 
that ${\cG}(z,\gl)$ is a proper convex and closed subset of $\C(K)$ with nonempty interior and  $\phi-\gl u_\lambda(z)\in\pl{\cG}(z,\gl)$.
By Hahn-Banach Theorem, there exists $\mu_\lambda\in\C(K)^*$, $\mu_\gl\not=0$, 
such that 
\[
\lan \mu_\lambda,\phi-\gl u_\lambda(z)\ran
\leq
\lan \mu_\lambda,f\ran
\quad
\FORALL f\in {\cG}(z,\gl).
\]
Being $\cG(z,\gl)$ a positive cone, this is only possible if $\lan \mu_\lambda,f\ran\geq 0$ for all $f\in\cG(z,\gl)$, and so $\min_{f\in\cG(z,\gl)} \lan \mu_\gl,f\ran=0$  in view of Lemma \ref{lemma 0 in cG}. From the fact that $\phi-\gl u_\lambda(z)\in\pl{\cG}(z,\gl)$, we conclude that   
\[
\lan \mu_\lambda,\phi-\gl u_\lambda(z)\ran=0.
\]
By Lemma \ref{lemma 0 in cG}, any nonnegative $\chi\in\C(K)$ is in ${\cG}(z,\gl)$. Hence,
\[
\lan\mu_\lambda,\chi\ran\geq 0
\quad
\FORALL \chi\in C_+(K),
\]
where $C_+(K):=\{g\in\C(K)\mid g\geq 0 \ \ON\  K\}$. 
By  Riesz Theorem, $\mu_\lambda$ can be identified with a positive Radon measure with compact support. 
By renormalization, we may assume that $\mu_\lambda$ is a probability measure. 
Thus $\mu_\lambda\in\parts(K)$ and satisfies 
\beq\label{eq25}
\gl u_\gl(z)
=
\int_{K} \phi(x,\xi)\,\mu_\lambda(dxd\xi),
\qquad
0
\leq
\lan
\mu_\gl,f
\ran
\quad
\FORALL 
f\in\cG(z,\gl),
\eeq
i.e., $\mu_\gl$ is the sought measure in ${\cG'}(z,\gl)$. 
\eproof

The corresponding statement when $\gl=0$ reads as follows.

\begin{theorem}\label{thm Mather measures} Let $\phi\in\C(K)$ and denote by  $c_\phi$ the critical constant associated with $H_\phi$ through \eqref{eq critical value}. Then 
there exists a measure $\mu_\lambda\in{\cGprime}(0)$
such that
\begin{equation}\label{claim discounted minimization}
-c_\phi=\int_K\phi(x,\xi)\,\mu_\lambda(dxd\xi)=\min_{\mu\in{\cGprime}(0)}\int_K \phi(x,\xi)\,\mu(dxd\xi).
\end{equation}
\end{theorem}

\bproof
By definition of $c_\phi$ and the fact that $H_{\phi+c_\phi+\eps}=H_\phi-(c_\phi+\eps)$, we have that $\phi+c_\phi+\eps\in \cG(0)$ for every $\eps>0$, yielding $\lan \mu, \phi+c_\phi+\eps\ran\geq 0$ for every $\mu\in\cG'(0)$ and $\eps>0$.  By arbitrariness of $\eps>0$, we get 
\[
\int_K \phi(x,\xi)\,\mu(dxd\xi)\geq -c_\phi\qquad\FORALL \mu\in {\cGprime}(0).
\]
Let us proceed to show the existence of a measure $\mu_0\in {\cGprime}(0)$ that 
realizes the equality in the above inequality. 
By Lemma \ref{lemma 0 in cG} and  Propositions \ref{prop cG convex} and \ref{prop nontrivial cone}, we know 
that $\overline{\cG}(0)$ is a proper closed and convex subset of $\C(K)$ with nonempty interior and  $\phi+c_\phi\in\pl{\cG}(0)$.
By Hahn-Banach Theorem, there exists $\mu_\lambda\in\C(K)^*$, $\mu_\gl\not=0$, 
such that 
\[
\lan \mu_\lambda,\phi+c_\phi\ran
\leq
\lan \mu_\lambda,f\ran
\quad
\FORALL f\in {\cG}(0).
\]
Being $\cG(0)$ a positive cone, this is only possible if $\lan \mu_\lambda,f\ran\geq 0$ for all $f\in\cG(0)$, and so $\min_{f\in\cG(0)} \lan \mu_\gl,f\ran=0$  in view of Lemma \ref{lemma 0 in cG}. From the fact that $\phi+c_\phi\in\pl{\cG}(0)$, we conclude that   
\[
\lan \mu_\lambda,\phi+c_\phi\ran=0.
\]
Now we proceed along the same lines of the proof of Theorem \ref{thm discounted Mather measures} to conclude that $\mu_\gl$ can be chosen in $\parts(K)$ and hence in $\cG'(0)$. 
\eproof

\section{Asymptotic analysis}\label{sec.asymptotic}

Let us consider a nonlocal HJ equation of the form 
\beq\label{eq31}
\gl u-\II u +H(x,Du)=0 \ \ \IN \T^d,
\tag{HJ$^\lambda_\M$}
\eeq
where $\gl\geq 0$, and $H\in\C(\TM)$ is a Hamiltonian which is convex in $p$, i.e., it satisfies conditions {\bf (H2)}. 
%
The nonlocal operator $\II$ is of the form \eqref{def I} with $M:=\M$  and conditions {\bf (N),\,(J1),\,(J2)} are in force. 
We will furthermore assume that the following facts hold:
\begin{itemize}
\item[\bf (EX)] {\bf (Existence of solutions)} The equation \eqref{eq31} admits solutions for every $\lambda\in [0,1)$.\smallskip
\item[\bf (LB)] {\bf (Lipschitz bounds)} There exists a constant $\kappa >0$ such that any continuous solution $u$ of \eqref{eq31} with $\lambda\in [0,1)$ satisfies  \ $|Du(x)|\leq \kappa \ \ \hbox{for a.e. $x\in\M$.}$
%
\end{itemize}
In Section \ref{sec.examples} we will present a rather general class of Hamiltonians for which the above stated assumptions are satisfied.

In this section we will study the asymptotic behavior of the (unique) solution $u_\lambda$ of \eqref{eq31} as $\lambda\to 0^+$.  
In the sequel, solutions, subsolutions and supersolutions of equation \eqref{eq31} with $\lambda=0$ will be termed {\em critical}. 

\subsection{Preliminaries} 
In order to exploit the results obtained in Section \ref{sec.duality}, we need a localization argument. 
To this aim, we choose a compact subset  $Q=Q(\kappa)$ of $\R^d$ such that
\begin{equation*} 
\pl_pH(x,p)\subset Q \quad\FORALL (x,p)\in\T^d\tim \ol B_{\kappa}. 
\end{equation*}
Let us set 
\[
\widetilde H(x,p):=\max\{H(x,p), |p|^2-k\}\quad\FORALL\ (x,p)\in\TM,
\]
with $k\in\N$ chosen big enough so that $H=\widetilde H$ on $\M\times \overline B_\kappa$.
%
%
Let us denote by $L$ the Lagrangian associated with $\widetilde H$ via the Fenchel transform, i.e.,
\[
L(x,\xi):=\sup_{p\in\R^d} [\langle p,\xi\rangle-\widetilde H(x,p)]\qquad\hbox{for all $(x,p)\in\TM$.}
\]
Since $\widetilde H$ has superlinear growth in $p$, the above supremum is a maximum, in particular it is finite. Furthermore, $L\in\C(\T^d\tim\R^d)$ and it is convex and superlinear in $\xi$, as it is well known.

Let us set 
\[
H_L(x,p):=\max_{\xi\in Q}[\langle p,\xi\rangle-L(x,\xi)]
\quad
\FORALL (x,p)\in\TM.
\]
The following holds. 

\begin{lemma}\label{lem.ma31} We have that \ $H(x,p)=H_L(x,p)$ \ for all  $(x,p)\in \M \tim \ol B_{\kappa}$. 
%
%
\end{lemma}
\bproof 
By continuity, it suffices to show that  $H=H_L$ on $\M\times B_\kappa$. Hence, let us fix $(x,p)\in \T^d\tim B_{\kappa }$. We have 
\begin{align*}
H(x,p)&=\widetilde H(x,p) = \sup_{\xi\in\R^d}[\langle p, \xi \rangle -L(x,\xi)]
=\!\!\max_{p\in\pl_\xi L(x,\xi)}[\langle p,\xi\rangle-L(x,\xi)]
=\!\!\max_{\xi\in\pl_p \widetilde H(x,p)}[\langle p,\xi\rangle-L(x,\xi)]
\\&
=\max_{\xi\in\pl_p H(x,p)}[\langle p,\xi\rangle-L(x,\xi)]
=\max_{\xi\in Q}[\langle p, \xi \rangle -L(x,\xi)]
=H_L(x,p), 
\end{align*}
where for the third to last equality we have used the fact that the identity  
$H=\widetilde H$ on $\M\times\ol B_\kappa$ implies in particular $\partial_p H=\partial_p \widetilde H$ on $\M\times B_\kappa$. 
\eproof

In view of the analysis performed in Section \ref{sec.duality}, we derive the following fact. 

\begin{proposition}\label{prop.equibounded}
The following holds:
\begin{itemize}
\item[\em (i)] for every fixed $\gl\in [0,1)$, any solution $u$ to \eqref{eq31} is also a solution to 
\beq\label{eq localized HJ}
\gl u-\II u +H_L(x,Du)=0 \ \ \IN \M.
\eeq
In particular, for every $\lambda\in (0,1)$, there exists a unique solution to \eqref{eq31} and it coincides with the unique solution to \eqref{eq localized HJ};\smallskip
\item[\em (ii)] the solutions $\{u_\lambda\mid\lambda\in (0,1)\}$ are equi-Lipschitz and equi-bounded. In particular, $\lambda\|u_\lambda\|_\infty\to 0$ as $\lambda\to 0^+$. 
\end{itemize}
\end{proposition}

\bproof
(i) Any solution $u_\lambda$ of \eqref{eq31} is also a solution of the same equation with $H_L$ in place of $H$ due to condition {\bf (LB)} and the fact that $H=H_L$ on $\M\times \overline B_{\kappa }$. The remainder of the assertion follows by applying Theorem \ref{lem.ma14} with $H:=H_L$.\\
\indent (ii) The equi-Lipschitz character of the functions  $\{u_\lambda\mid\lambda\in (0,1)\}$ holds by assumption (LB). Let us show that they are equi-bounded. Pick a critical solution $u$, i.e., a solution of \eqref{eq31} with $\lambda=0$. Choose a sufficiently large $k\in\N$ such that $\underline u:=u-k\leq 0\leq u+k=\overline u$ on $\M$.  It is easily seen that the functions $\underline u,\, \overline u$ are critical solutions as well. Furthermore,  due to their sign, they are,  respectively, a sub and a super solution of \eqref{eq31} for every $\lambda>0$, and hence of \eqref{eq localized HJ}  due to condition {\bf (LB)}.  By the Comparison Principle holding for equation  \eqref{eq localized HJ}, see Proposition \ref{prop comparison}, we infer that $\underline u\leq u_\lambda\leq \overline u$ in $\M$ for every $\lambda>0$. 
\eproof

\subsection{Mather and discounted Mather measures} We put ourselves in the setting presented in Section \ref{sec.duality} and we 
adopt the notation therein introduced, where we keep denoting by $Q=Q(\kappa)$ 
and $L$ the compact set and the Lagrangian introduced in the previous subsection. Let us set $K:=\M\times Q$.

For every fixed $\gl> 0$ and $z\in\M$, by definition of $\cG'(z,\gl)$ we have 
\[
\int_K L(x,\xi)\,\mu(dxd\xi)
\geq
\gl u_\gl(z)
\qquad
\FORALL
\mu\in\cG'(z,\gl).
\]
The existence of probability measures in $\cG'(z,\gl)$ that minimize the integral of $L$
is guaranteed by Theorem \ref{thm discounted Mather measures}. 
The precise statement is the following.

\begin{theorem}\label{thm.discounted Mather measures}Let $(z,\gl)\in \T^d\tim (0,+\infty )$ and $u_\gl$ 
be the unique solution to \eqref{eq31}. Then there exists $\mu_\gl\in\cG'(z,\gl)$ such that
\beq\label{claim discounted minimization 2}
\gl u_\gl(z)=\int_K L(x,\xi)\,\mu_\gl(dxd\xi)=\min_{\mu\in \cG'(z,\gl)}\int_K L(x,\xi)\,\mu(dxd\xi).
\eeq
\end{theorem}

A measure $\mu\in \cG'(z,\gl)$ that satisfies \eqref{claim discounted minimization 2}  will be called {\em discounted Mather measure} for $L$. The set of discounted Mather measures will be denoted by $\Mis(z,\lambda,L)$ in the sequel.\smallskip

We proceed to show the existence of measures that solve the minimization problem \eqref{claim discounted minimization 2} for $\lambda=0$. 
This is guaranteed by Theorem \ref{thm Mather measures} after noticing that the critical constant $c$ associated with $H_L$ via \eqref{eq critical value} is equal to $0$ in view of condition {\bf (EX)} and Proposition \ref{prop LPV}.
We recall that the corresponding set  $\cG'(z,0)$ is independent of $z$ and will be  more simply denoted by  $\cG'(0)$.

\begin{theorem}\label{thm.Mather measures 1}
There exists $\mu\in\cG'(0)$
such that
\begin{equation}\label{claim Mather minimization}
0=\int_K L(x,\xi)\,\mu(dxd\xi)=\min_{\mu\in\cG'(0)}\int_K L(x,\xi)\,\mu(dxd\xi).
\end{equation}
\end{theorem}

A measure $\mu\in\cG'(0)$ that satisfies \eqref{claim Mather minimization}  will be called {\em Mather measure} for $L$. 
The set of Mather measures will be denoted by $\Mis(L)$ in the sequel.

We record here, for further use, the following result. 

\begin{lemma}\label{lem.ma36}
Let $z\in\M$ and $(\lambda_n)_n$ be an infinitesimal sequence in $(0,+\infty)$. For each $n\in\N$, let $\mu_n\in\Mis(z,\lambda_n,L)$. Then there exists a subsequence $(\mu_{n_k})_k$ that weakly converges to a probability measure $\mu$ on $K$, i.e. 
\[
\lim_{k} \int_K f(x,\xi)\,\mu_{n_k}(dxd\xi)
=
\int_K f(x,\xi)\,\mu(dxd\xi)
\qquad
\FORALL f\in\C(K).
\]
Furthermore, $\mu\in\Mis(L)$. 
\end{lemma}

\bproof
Being $K$ compact, up to extracting a subsequence (not relabeled), the measures  $(\mu_n)_n$ weakly converge to a probability measure $\mu$ on $K$.   In view of Proposition \ref{prop.equibounded}-(ii), we have 
\beq\label{eq minimal}
\int_K L(x,\xi)\,\mu(dxd\xi)
=
\lim_{n} \int_K L(x,\xi)\,\mu_n(dxd\xi)
=
\lim_{n} \lambda_n u_{\lambda_n}(z)
=0.
\eeq
It is left to show that $\mu\in\cG'(0)$. To this aim, let us fix $(\phi,u)\in\cF(0)$. It is easily seen that $(\phi+\gl_n u,u)\in\cF(\gl_n)$ and so $\phi+\gl_n u-\gl_n u(z)\in \cG(z,\gl)$ for every fixed $z\in\M$ and for all $n\in\N$, in particular 
\[
\lan \mu_n, \phi+\gl_n u-\gl_n u(z)\ran \geq 0
\qquad
\FORALL n\in\N.
\]
Sending $n\to +\infty$, we get\ \ $\lan \mu,\phi\ran\geq 0$. This shows that $\mu\in\cG'(0)$ by the arbitrariness of the choice of $(\phi,u)\in\cF(0)$,  and hence that 
that $\mu\in\Mis(L)$ in view of \eqref{eq minimal}. 
\eproof

\subsection{Asymptotic convergence.}\label{sec.convergence} In this subsection we aim to show that the solutions $u_\lambda$ of \eqref{eq31} converge, as $\lambda\to 0^+$, to a specific critical solution $u_0$. The first step consists in identifying a good candidate for $u_0$. 
To this aim, we consider the family $\sol_-$ of critical solutions  $u:\M\to \R$ 
satisfying the following condition
\begin{equation}\label{CONDITiON U0}
\int_K u(x)\,\mu(dxd\xi)\leq 0 \qquad\text{for every $\mu\in\Mis(L)$}.
\end{equation}
We note that $\sol_-$ is nonempty. Indeed, $u-\lVert u\rVert_\infty$ is in $\sol_-$ whenever $u$ is a critical solution. 

\begin{lemma}\label{lemma F meno}
The family $\sol_-$ is uniformly bounded from above, i.e.
$$\sup\{u(x)\mid x\in \M,u\in \sol_-\}<+\infty.$$
\end{lemma}
\begin{proof} 
According to the assumption {\bf (LB)}, the family of critical solutions is equi--Lipschitz, with $\kappa $ as a common Lipschitz constant. Pick $\mu\in\Mis(L)$. For any $u\in\sol_-$, we have 
\[
\min_\M u=\int_K \min_\M u\ d\mu\leq \int_K u\,d\mu\leq 0.
\]
We infer \quad $\max_\M u\leq \max_\M u-\min_\M u\leq \kappa  \operatorname{diam}(\M)=\kappa \sqrt d$.
 \end{proof}
Therefore we can define  $u_0:\M\to \R$ by
\begin{equation}\label{first def u_0}
u_0(x):=\sup_{u\in\sol_- }u(x)\qquad\hbox{for all $x\in\M$}.
\end{equation}
As the supremum of a family of equi--Lipschitz critical solutions, we derive that $u_0$ is itself a Lipschitz critical subsolution, see Proposition \ref{prop sup inf}. We will obtain later, as a consequence of the asymptotic convergence, that $u_0$ is an element of $\sol_-$, see Theorem \ref{thm main} below.

We now start to study the asymptotic behavior of the discounted solutions $u_\lambda$ as $\lambda\to 0^+$ and the relation with $u_0$.  We begin with the following result:

\begin{proposition}\label{prop.ineq lim}
For every fixed $\lambda>0$, we have 
\begin{equation}\label{eq ineq lim}
\lan \mu, u_\gl\ran
=
\int_K u_\lambda(x)\, \mu(dxd\xi)\leq 0\qquad\hbox{for every $\mu\in\Mis (L)$.} 
\end{equation}
In particular, if the functions $u_{\lambda_n}$ uniformly converge to $u$ for some sequence $\lambda_n\to 0^+$, then 
$u\in\sol_-$ and hence $u\leq u_0$ on $\M$.
\end{proposition}

\begin{proof} 
Let us fix $\mu\in\Mis(L)$. Since $(L-\gl u_\gl,u_\gl)\in\cF(0)$, we find that 
\[
0
\leq 
\lan \mu, L-\gl u_\gl\ran
=
\lan \mu, L\ran 
-
\gl \lan \mu,u_\gl\ran 
=
-\gl \lan \mu, u_\gl\ran,
\]
and thus $\lan \mu, u_\gl\ran \leq 0$. 
If $u$ is the uniform limit of $\left(u_{\lambda_n}\right)_n$ for some $\lambda_n\to 0^+$, then $u$ is a critical solution and satisfies $\lan \mu, u\ran \leq 0$ for every $\mu\in\Mis(L)$ in view of \eqref{eq ineq lim}.  
Therefore  $u\in \sol_-$ and $u\leq u_0$. 
 \end{proof}

The next (and final) step is to show that $u\geq u_0$ in $\M$ whenever $u$ is the uniform limit of  $\left(u_{\lambda_n}\right)_n$ for some $\lambda_n\to 0^+$.

The following lemma will be crucial for the proof of the convergence result, see Theorem \ref{thm main} below.
\begin{lemma}\label{ineq prim}
Let $(z,\gl)\in \T^d\tim (0,+\infty )$ and $\mu_\lambda\in\Mis(z,\gl,L)$ be a discounted Mather measure. For any 
critical subsolution $v\in\C(\M)$, we have
\begin{equation}\label{eq pre-final}
u_\lambda(z)\geq v(z)-\int_{K} v(x)\, {\mu}_{\lambda} (dxd\xi).
\end{equation}
\end{lemma}

\begin{proof}
Let $v\in\C(\M)$ be a critical subsolution. We have $(L+\gl v,v)\in\cF(\gl)$, hence 
$L+\gl v-\gl v(z)\in\cG(z,\gl)$. We derive  
\[
0
\leq 
\lan L+\gl v-\gl v(z),\mu_\gl\ran
=
\lan L,\mu_\gl\ran +\gl \lan v, \mu_\gl\ran -\gl v(z)
=
\gl \big(u_\gl (z) +\lan v, \mu_\gl\ran -v(z)\big).
\] 
By dividing the above inequality by $\gl >0$, we get\ \ 
$u_\gl(z) \geq v(z)-\lan v, \mu_\gl\ran$,
\ 
namely \eqref{eq pre-final}. 
\end{proof}

We are now ready to prove the convergence result. 

\begin{theorem}\label{thm main}
The functions $\{u_\lambda\mid\gl\in (0,1)\}$ uniformly converge to  $u_0$ on $\M$ {as $\lambda\to 0^+$ }.  
Furthermore, $u_0$ belongs to $\sol_-$ and 
\[
\max_{\mu\in\Mis(L)}\int_K u_0(x)\, \mu(dxd\xi)=0.
\]
\end{theorem}

\begin{proof}
By assumption {\bf (LB)} and Proposition \ref{prop.equibounded}, we know that the functions $u_\lambda$ are equi--Lipschitz and equi--bounded, 
hence it is enough, by the Ascoli--Arzel\`a theorem, to prove that any converging subsequence has  $u_0$ as limit.

Let $\lambda_n\to 0^+$ be such that $u_{\lambda_n}$ uniformly converges to some $u\in \C(\M)$. 
We have seen in Proposition \ref{prop.ineq lim} that\  \ 
$u(x)\leq u_0(x)\ \ \hbox{for every  $x\in \M$}.$
To prove the opposite inequality, let us fix $z\in \M$.  Let $v$ be a critical subsolution.
By Lemma \ref{ineq prim}, we have
\begin{equation*}
u_{\lambda_n}(z)\geq v(z)-\int_{K} v(x)\, {\mu}_{\lambda_n} (dxd\xi).
\end{equation*}
By Lemma \ref{lem.ma36}, up to extracting a further subsequence (not relabeled), we can assume that
the measures ${\mu}_{\lambda_n}$ weakly converge to a Mather measure $\mu_0\in\Mis(L)$. 
Passing to the limit in the last inequality,
we get  
\begin{equation}\label{eq.2nd inequality}
u(z)
\geq 
v(z)-\int_{K} v(x)\, {\mu_0} (dxd\xi).
\end{equation}
If we furthermore assume that 
$v\in \sol_-$, the set of solutions satisfying \eqref{CONDITiON U0}, we obtain 
\[
u(z)\geq w(z).
\] By the arbitrariness of the choice of $z\in\M$ and by definition of $u_0$, we conclude that \ \ $u(z)\geq u_0(z)=\sup\limits_{v\in\sol_-}v(z)$\ \ for all $z\in\M$.

As an accumulation point of the functions $\{u_\lambda\,\mid\,\lambda>0\,\}$ as $\lambda\to 0^+$, the function $u_0$ is a critical solution. The fact that $u_0\in\sol_-$ directly follows from Proposition \ref{prop.ineq lim}. By picking $v:=u_0$ in \eqref{eq.2nd inequality}, we infer that $\int_K u_0(x)\, \mu_0(dxd\xi)\geq 0$ for some $\mu_0\in\Mis(L)$, giving the last assertion.
\end{proof}

\section{Applications}\label{sec.examples}

In this section we present a class of convex and superlinear Hamiltonians to which we can apply the results established in Section \ref{sec.asymptotic}. 

To this aim, we introduce the class $\Ham$ made up of functions $G:\TR\to\R$ that are convex in $p$ and satisfy the following conditions, for some constants $\alpha_0,\alpha_1>0$ and $\gamma>1$:
	\begin{itemize}
		\item[\bf (G1)] \quad $\alpha_0|p|^\gamma-1/\alpha_0\leqslant G(x,p)\leqslant\alpha_1(|p|^\gamma+1)$\quad for all ${(x,p)\in\TR}$;\medskip
		\item[\bf (G2)]\quad $|G(x,p)-G(x,q)|\leqslant\alpha_1\left(|p|+|q|+1\right)^{\gamma-1}|p-q|$\quad for all $x\in\R^d$ and $p,q\in\R^d$;\medskip
	\item[\bf (G3)] \quad $|G(x,p)-G(y,p)|\leqslant\alpha_1(|p|^\gamma+1)|x-y|$\quad 
	for all $x,y\in\R^d$ and $p\in\R^d$.\medskip		
\end{itemize}
We will denote by $\Hamper$ the family of Hamiltonians $G$ in $\Ham$ that are $\Z^d$--periodic in $x$. 

We are going to show that, for every fixed $G\in\Hamper$, there exists a unique real constant $c$ such that the Hamiltonian $H(x,p):=G(x,p)-c$ satisfies conditions {\bf (EX)} and {\bf (LB)} in Section \ref{sec.asymptotic}. In order to do this, we will exploit the results proved in \cite{BLT17}, which apply in particular to a Hamiltonian $G$ satisfying conditions {\bf (G2)}-{\bf (G3)} above and the following one, for some constants $b,K>0$ and $\gamma>1$: 
\begin{itemize}
\item[\bf (BLT)]\quad  
$\mu G(x,p/\mu)-G(x,p) \geqslant (1-\mu)(b|p|^\gamma-K)
\quad
\hbox{for all $\mu\in (0,1)$ and $(x,p)\in\TR$.}$
\end{itemize}

We are going to show that the convexity of $G$ in $p$ and condition {\bf (G1)} above actually imply  {\bf (BLT)}, for the same constant $\gamma> 1$ and for some constants $b,K>0$ only depending on  $\alpha,\beta$ and $\gamma$. In fact, this is a direct consequence of the following two lemmas. 

\begin{lemma}\label{lemma growth convex}
Let $G:\TR\to\R$ be a convex Hamitonian satisfying the growth conditions stated in {\bf (G1)}. Then $G$ satisfies 
\begin{equation}\label{eq Barles equivalent}
\langle q,p\rangle -G(x,p)
\geqslant 
b |p|^\gamma-K
\qquad
\hbox{for all $(x,p)\in\TM$ and $q\in\partial_p G(x,p)$.}
\end{equation}
for the same constant $\gamma> 1$ and for some constants $b,K>0$ only depending on  $\alpha,\beta$ and $\gamma$.
\end{lemma}

\begin{proof}
For $(x,p)\in\R^d\times\R^d$, $t\in(0,1)$, and $q\in\partial_pG(x,p)$, we have by convexity
\[
G(x,tp)\geq G(x,p)+\du{q, tp-p}=G(x,p)-(1-t)\du{q,p},
\]
which reads
\[
(1-t)(\du{q,p}-G(x,p))\geq tG(x,p)-G(x,tp).
\]
By {\bf (G1)}, we obtain
\[\bald
tG(x,p)-G(x,tp)&\geq t(\ga|p|^\gamma-1/\ga)-\gb(t^\gamma |p|^\gamma+1)
\\& =(\ga t-\gb t^\gamma)|p|^\gamma-t/\ga -\gb.
\eald
\]
Thus,
\[
(1-t)(\du{q,p}-G(x,p))\geq (\ga t-\gb t^\gamma)|p|^\gamma-t/\ga-\gb.
\]
We choose $\tau>0$ so that
\[
2\gb \tau^\gamma=\ga \tau,
\] 
which can be solved as
\[
\tau=\left(\fr{\ga}{2\gb}\right)^{\fr{1}{\gamma-1}}. 
\]
Noting that $0<\tau<1$ (notice that $\gb\geq \ga$), 
we find that
\[
\du{q,p}-G(x,p)\geq \fr{\ga\tau}{2(1-\tau)}|p|^\gamma - \fr{\tau+\ga\gb}{\ga(1-\tau)}. 
\]
Setting 
\[
b=\fr{\ga\tau}{2(1-\tau)} \ \ \AND \ \ K=\fr{\tau+\ga\gb}{\ga(1-\tau)}.
\]
we have
\[
\du{q,p}-G(x,p)\geq b|p|^\gamma-K\quad\hbox{for all  $x,p\in\R^d$} \ \AND\ q\in\partial_pG(x,p). \qedhere
\]
\end{proof}

\begin{lemma}\label{lemm convex Barles condition}
Let $G:\R^d\times\R^d\to\R$ be a convex Hamiltonian satisfying assumption \eqref{eq Barles equivalent} for some constants  $b,K>0$ and $\gamma>1$.
Then $G$ satisfies condition  {\bf (BLT)}, with same constants  $b,K>0$ and $\gamma>1$. 
\end{lemma}

\begin{proof}
Let us fix $(x,p)\in\TR$. 
By convexity, for every fixed $q\in\partial_p G(x,p)$ we have 
\[
G(x,\hat p) \geqslant G(x,p)+\langle q,\hat p-p\rangle
\qquad
\hbox{for all $\hat p\in\R^d$.}
\]
Fix $\mu\in (0,1)$ and  plug $\hat p:=p/\mu$ in the above inequality. We have 
\[
G(x,p/\mu)\geqslant G(x,p)+\langle q,\frac{1-\mu}{\mu} p),
\]
yielding 
\[
\mu G(x,p/\mu) 
\geqslant 
\mu G(x,p)+(1-\mu)\langle q,p\rangle 
\underset{\eqref{eq Barles equivalent}}{\geqslant} 
G(p)+(1-\mu)(b|p|^\gamma-K).\qedhere
\]
\end{proof}

We proceed by showing existence, uniqueness and the Lipschitz character of the solutions of the discounted nonlocal HJ equation associated with a Hamiltonian $G$ in $\Hamper$. 

\begin{theorem}\label{thm discounted functional}
Let $G\in\Hamper$. For every $\lambda>0$, there exists a unique viscosity solution $\hat u_\lambda$ of the following nonlocal HJ equation
\beq\label{eq discounted functional}
\lambda u-\II u +G(x,Du)=0\qquad\hbox{in $\M$.}
\eeq
Furthermore, the family of solutions $\{\hat u_\lambda\,\mid\,\lambda>0\}$ is equi-Lipschitz. 
\end{theorem}

\bproof
The existence and uniqueness of the solution to \eqref{eq discounted functional} follows from \cite[Proposition 2.6]{BLT17}. 
The equi-Lipschitz character of the solutions $\hat u_\lambda$ follows from Theorem 3.1 and Lemma 4.2 in \cite{BLT17}. 
\eproof

The proof of Lemma 4.2 in \cite{BLT17} is based on an inequality, see (4.4) in \cite{BLT17}, that is stated without proof. We provide it here for the reader's convenience. 

\begin{lemma}\label{lemma app Barles claim}
Let $G\in\C(\TR)$ be a Hamiltonian satisfying condition  {\bf (BLT)}. For every $\delta>0$, 
there exists $k(\delta,\gamma,b)\in\N$ such that,  for every $(x,p)\in\TR$ with $|p|\geqslant \delta$, the following inequality 
holds
\begin{equation}\label{claim Barles}
G(x,tp)-tG(x,p)
\geqslant 
\eta |t|^\gamma |p|^\gamma
\qquad
\hbox{for all $t=2^k$ with $k\geqslant k(\delta,\gamma,b)$,}
\end{equation}
where $\eta=b/2^{\gamma+1}$. 
\end{lemma}

\begin{proof}
Let us fix $x,p\in\R^d$ with $|p|\geqslant \delta$.
By setting $\mu:=1/t$ with $t>1$ in  {\bf (BLT)}, we get 
\begin{equation}\label{eq Barles condition2}
G(x,tp)\geqslant tG(x,p) \geqslant (t-1)(b|p|^\gamma-K)\qquad\hbox{for all $t>1$.}
\end{equation}
Let us plug $t:=2^k$ in \eqref{eq Barles condition2} with $k\in\N$. To ease notation, we will neglect the dependence of $G$ on $x$ since the latter will remain fixed throughout the proof. 
\begin{itemize}
\item {$\mathbf{t=2}$.} \quad $G(2p)\geqslant 2G(p)+b|p|^\gamma-K$.\smallskip
\item {$\mathbf{t=2^2}$.} By making use of the inequality just obtained for $t=2$, we get 
\begin{eqnarray*}
G(4p)
&\geqslant& 
2G(2p)+b2^\gamma|p|^\gamma-K\\
&\geqslant& 
2(2G(p)+{b|p|^\gamma-K}+(b2^\gamma |p|^\gamma-K)\\
&\geqslant& 
4G(p)+b2^\gamma |p|^\gamma-(2^2-1)K.\smallskip
\end{eqnarray*}
\item {$\mathbf{t=2^3}$.} By making use of the inequality just obtained for $t=2^2$, we get 
\begin{eqnarray*}
G(8p)
&\geqslant& 
2G(4p)+b(2^2)^\gamma|p|^\gamma-K\\
&\geqslant& 
2(4G(p)+b2^\gamma|p|^\gamma-(2^2-1)K)-K+b(2^2)^\gamma |p|^\gamma\\
&\geqslant& 
2^3G(p)+b2^\gamma |p|^\gamma-(2^3-1)K.\smallskip
\end{eqnarray*}
\item {$\mathbf{t=2^k}$.} By iteration, we get 
\begin{eqnarray*}
G(2^kp)
&\geqslant& 
2^kG(p)+b(2^{k-1})^\gamma|p|^\gamma-(2^k-1)K.
\end{eqnarray*}
\end{itemize}
In order to obtained the asserted inequality \eqref{claim Barles}, we need to control the dangerous term $(2^k-1)K$ which appears in the last inequality. To this aim, we notice that   
\begin{eqnarray*}
b(2^{k-1})^\gamma|p|^\gamma
\geqslant
\frac{b}{2^{\gamma+1}}(2^{k})^\gamma|p|^\gamma
+
\frac{b}{2^{\gamma+1}}(2^{k})^\gamma|\delta|^\gamma
\qquad
\hbox{for all $|p|\geqslant \delta$,}
\end{eqnarray*}
hence 
\begin{eqnarray*}
G(2^kp)
\geqslant 
2^kG(p)
+
\frac{b}{2^{\gamma+1}}(2^{k})^\gamma|p|^\gamma
+
2^k\underbrace{\left(\frac{b|\delta|^\gamma}{2^{\gamma+1}}(2^{k})^{\gamma-1}-K\right)}_{D}.
\end{eqnarray*}
By choosing $k(\delta,\gamma,b)$ big enough so that $D\geqslant 0$, we get the assertion. 
\end{proof}

Next, we analyze the solvability of the nonlocal critical HJ equation associated with a  $G$ belonging to $\Hamper$. 
The real constant $c$ appeared below is known as {\em ergodic} or {\em critical} constant. 

\begin{theorem}\label{thm critical functional}
Let $G\in\Hamper$. There exists a unique constant $c$ in $\R$ such that the following nonlocal HJ equation
\beq\label{eq functional critical}
-\II u +G(x,Du)=c\qquad\hbox{in $\M$}
\eeq
admits solutions. Furthermore, the family of solutions to \eqref{eq functional critical} is equi-Lipschitz. 
\end{theorem}

\bproof
The existence and uniqueness of such a real constant $c$ follows from \cite[Proposition 4.1]{BLT17}. The equi-Lipschitz character of the solutions $u$ to \eqref{eq functional critical} follows from Theorem 3.1 and Lemma 4.2 in \cite{BLT17}. 
\eproof

By gathering the information obtained above, we derive the following fact. 

\begin{proposition}\label{prop application}
Let $G\in\Hamper$ and denote by $c\in\R$ the associated critical constant. Then conditions {\bf (EX)} and {\bf (LB)} hold for the Hamiltonian $H$ defined as $H(x,p):=G(x,p)-c$ for all $(x,p)\in\TM$. 
\end{proposition}

\bproof
For every fixed $\lambda>0$, the function $\hat u_\lambda+c/\lambda$ solves the equation \eqref{eq31}. 
Conversely, any solution $u_\lambda$ to \eqref{eq31} is such that $u_\lambda-c/\lambda$ solves the equation 
\eqref{eq discounted functional}, hence it is unique by Theorem \ref{thm critical functional} and $u_\lambda=\hat u_\lambda+c/\lambda$ in $\M$. 
The assertion follows as a direct consequence of Theorems \ref{thm discounted functional} and \ref{thm critical functional}. 
\eproof

\begin{appendix}

\section{Some technical lemmas}

In this section we collect some technical lemmas that we employ. We begin with the following. 

\begin{lemma}\label{approx} Let $v\in\Bd(\R^d)\cap\USC(\R^d)$. 
Then there is a sequence $\{v_k\mid k\in\N\}\subset \C_{\mrm{b}}^\infty(\R^d)$ such that 
\[
v_k(x)\geq v_{k+1}(x) \ \ \FOR (x,k)\in\R^d\tim\N \ \ \AND \ \ \lim_k v_k(x)=v(x) \ \ \FOR x\in\R^d.
\]
\end{lemma} 

\bproof  Consider the sup-convolution $v^\eps$ of $v$ defined according to \eqref{def sup convolution}.  
Fix a decreasing sequence $\{\ep_k\mid k\in\N\}\subset (0,+\infty)$ converging to zero. 
Set 
\[
w_k(x)=v^{\ep_k}(x)+2^{-k} \ \ \FOR (x,k)\in \R^d\tim\N.
\]
In view of Proposition \ref{properties sup convolution} we have 
\beq\label{eq1.0}
\lim_k w_k(x)=v(x) \ \ \FOR x\in\R^d \ \ \AND \ \ w_k(x)>w_{k+1}(x) \ \ \FOR (x,k)\in \R^d\tim\N.
\eeq
Furthermore, note that 
\beq\label{eq1.1}
w_k(x)-w_{k+1}(x)\geq 2^{-k}-2^{-k-1}=2^{-k-1} \ \ \FOR (x,k)\in\R^d\tim\N.
\eeq
Let $\rho_\gd$ be a standard mollification kernel, with $\supp \rho_\gd\subset B_\gd$. 
Since 
\[
\lim_{\gd\to 0^+}w_k*\rho_\gd =w_k \ \ \text{ uniformly in }\R^d,
\]
we may choose $\gd_k>0$ so that
\beq\label{eq1.2}
\|w_k*\rho_{\gd_k}-w_k\|_\infty \leq 2^{-k-2}. 
\eeq
We set 
\[
v_k=w_k*\rho_{\gd_k}. 
\]
Observe that $v_k\in \C_{\mrm{b}}^\infty(\R^d)$ for every $k\in\N$, and that, for all $(x,k)\in\R^d\tim\N$,
\[\bald
v_{k}(x)-v_{k+1}(x) &
\underset{\eqr{eq1.2}}{\geq}		
w_{k}(x)-2^{-k-2}-w_{k+1}(x)-2^{-k-3}\\
& \underset{\eqr{eq1.1}}{>}	w_k(x)-w_{k+1}(x)-2^{-k-1}\geq 0.
\eald
\]
Furthermore, from  \eqr{eq1.0} and \eqr{eq1.2} we get 
\[
\lim_k v_k(x)=v(x)\qquad\FORALL x\in\R^d. \qedhere
\] 
\eproof 

Next, we state a sort of global version of Jensen's lemma, see \cite[Lemma A.3]{users}.

\begin{proposition} \label{max1}
Let $u$ be a bounded and semiconvex function on $\R^d$ and $\varphi
\in \C_\mathrm{b}^2(\R^d)$ such that $u-\varphi$ has a strict maximum point at $\hx$. 
Let $R>r>0$ be fixed. 
Then there is a sequence $(\varphi_k,x_k)\in \C^2(\R^d)\tim \R^d$ such that 
for all $k\in\N$, 
\[\left\{\,\bald
&x_k\in B_r(\hx) \ \ \FOR k\in\N,\\
&\varphi_k(x)=\varphi(x) \ \ \FOR x\in\R^d\stm B_R(\hx), \\
&\max\limits_{\Rd}(u-\varphi_k)=(u-\varphi_k)(x_k) \ \ \FOR k\in\N,\\
&u \text{ is twice differentiable at } x_k,\\
\eald \right.
\]
and $\varphi_k \to \varphi \ \IN \C^2(\R^d)$ as $k\to +\infty$.
\end{proposition}

\bproof  By translation, we may assume $\hx=0$.  We may also assume that 
$(u-\varphi)(0)=0$. By assumption, there exists $\gd>0$ such that
\beq\label{strict}
\max_{\ol{B}_R\stm B_{r}}(u-\varphi)<-\gd. 
\eeq
Choose a function $\chi\in \C^2(\R^d)$ such that 
\[\bald
&0\leq \chi(x)\leq 1 \ \ \FOR x\in\R^d,
\\&
\chi(x)=\bcases 0 &\FOR x\in \R^d\stm B_R,\\[3pt]
1 \ \ &\FOR x\in B_r. 
\ecases
\eald
\]
Since the function $u$ is semiconvex, the argument employed in \cite[Lemma A.3]{users} to prove Jensen's lemma  actually shows that there exist sequences $(x_k)_k\subset B_r\,$ and $(p_k)_k\subset \R^{d}$
with 
\beq\label{max0}
\lim_{k\to+\infty}x_k=\lim_{k\to+\infty}p_k=0
\eeq
such that the function 
$\Phi_k(x):=(u-\varphi)(x)+\langle p_k, x\rangle$ is twice differentiable at $x_k$ and 
\beq\label{max2+1}
\Phi_k(x_k)=\max_{\ol B_r}\Phi_k.
\eeq
Notice that this means that $u$ is twice differentiable at $x_k$. 
By setting 
\[
\psi_k(x):=\varphi(x)-\langle p_k, x\rangle +\Phi_k(x_k),
\]
we have 
\[
u(x)\leq \psi_k(x) \ \ \FORALL x\in \ol B_r,
\quad
u(x_k)=\psi_k(x_k).
\]
Note that 
\[\bald
&\lim_{k\to +\infty}\Phi_k(x_k) = (u-\varphi)(0)=0,
\\& \lim_{k\to +\infty}\psi_k=\varphi \ \ \IN \C^2(\R^d).
\eald
\]
Thanks to \eqref{strict} and the fact that 
\[
\lim_{k\to+\infty}(u-\psi_k)=(u-\varphi) \ \ \IN \C(\R^d),
\]
we may assume, by possibly passing to a subsequence (not relabeled), that, 
for every $k\in\N$, 
\[
u\leq \psi_k \ \ \FORALL x\in\ol B_R\stm B_r,
\]
yielding 
\[
u\leq \psi_k \ \ \FORALL x\in B_R. 
\] 
To resume, we have
\[\bald
&u(x_k)=\psi_k(x_k) \ \ \FORALL k\in\N,\\
&u\leq \psi_k \ \ \FORALL x\in B_R\ \AND\  k\in\N,\\
&u\leq \varphi \ \ \ \FORALL x\in\R^d\ \AND\ k\in\N.
\eald
\] 
We define $\varphi_k\in \C_\mathrm{b}^2(\R^d)$ as \  $\varphi_k=(1-\chi)\varphi+\chi \psi_k$.
It is easily seen that 
\[\bald
&u(x_k)=\varphi_k(x_k)  && \FORALL k\in\N,\\
&u\leq \varphi_k &&\FORALL x\in\R^d\ \AND\ k\in\N,\\
&\varphi_k(x)=\varphi(x)  &&\FORALL x\in\R^d\stm B_R\ \AND\ k\in\N.
\eald
\] 
Since 
\[
\lim_{k\to+\infty}\psi_k=\varphi \ \ \ \IN \C^2(\R^d),
\]
we easily deduce that 
\[
\lim_{k\to +\infty}\varphi_k=\varphi \ \ \ \IN \C^2(\R^d).\qedhere
\]
\eproof

\section{Approximation of subsolutions via sup-convolutions}\label{app.sup convolution}

Let us consider the following nonlocal HJ equation 
\begin{equation}\label{eq0}
\gl u-\II u+H (x,Du)=0 \ \ \ \IN \T^d \tag{HJ$^\lambda_\M$},
\end{equation}
where $\lambda\geq 0$, and $H\in\C(\TM)$ satisfies {\bf (H1)} and is additionally assumed convex in $p$, i.e. it satisfies {\bf (H2)}.
%
%
The nonlocal differential operator $\II$ is of the form \eqref{def I}  and conditions {\bf (N), (J1), (J2)} are in force with $M:=\M$. 

In this section we are interested in showing that any continuous subsolution $v$ of \eqref{eq0} can be approximated, in the sup norm, by a family of semiconvex functions that are subsolutions to \eqref{eq0}, up to a small error. 
To this aim, we recall the definition of sup-convolution of an upper semicontinuous and bounded function  and its main properties.  

\begin{proposition}\label{properties sup convolution}
Let $v\in\Bd(\R^d)\cap\USC(\R^d)$. For $\eps>0$, let
\beq\label{def sup convolution}
v^\ep(x):=\sup_{y\in\R^d} \left(v(y)-\fr{1}{2\ep}|x-y|^2\right) \ \ \FOR x\in\R^d
\eeq
be the sup-convolution of $v$. 
The following holds:
\begin{itemize}
\item[\em (i)]\quad $v^{\ep}(x)\geq v^{\gd}(x)\geq v(x) \ \ \IF \ep>\gd>0$ \quad and \quad
$\lim_{\ep\to 0^+}v^\ep(x)=v(x)$ \quad for all $x\in\R^d$;\smallskip
\item[\em (ii)] \quad $v^\eps$ is $1/\eps$-semiconvex in $\R^d$;\smallskip
\item[\em (iii)] \quad $v^\eps$ is bounded and Lipschitz in $\R^d$. 
\end{itemize}
\end{proposition}

The precise statement that we are going to prove is the following.

\begin{proposition}\label{prop sup convolution} Assume {\bf (N),\,(J1),\,(J2)} and {\bf (H1)-(H2)}. Let 
$v\in \C(\T^d)$ be a subsolution to  
\eqref{eq0} for some $\lambda\geq 0$  and  
\[
v^\ep(x):=\max_{y\in\R^d}\Big(v(y)-\fr{1}{2\ep}|x-y|^2\Big),\qquad x\in\R^d,
\]
be the sup-convolution of $v$. Then, for every $\delta>0$, there exists 
$\eps(\delta) >0$ such that the function $v^{\eps(\delta)}$ is a 
Lipschitz subsolution to 
\begin{equation}\label{eq eps}
\gl u-\II u+H (x,Du)\leq \delta \ \ \ \IN \T^d,
\end{equation}
satisfying \ \ $\|v-v^{\eps(\delta)}\|_\infty<\delta$. 
\end{proposition}

\bproof 
The fact that $v^\eps$ is $\Z^d$--periodic and Lipschitz continuous, for every $\eps>0$, is easily checked.  
Let $\varphi\in \C_{b}^2(\R^d)$ be such that 
$(v^\ep-\varphi)(\hx)=\max(v^\ep-\varphi)$ at some $\hx\in\R^d$. Choose $\hy\in\R^d$ so that 
\[
v^\ep(\hx)=v (\hy)-\fr{1}{2\ep}|\hx-\hy|^2. 
\]
Since $v(\hx)\leq v^\ep(\hx)$, it follows that 
\[
\fr{1}{2\ep}|\hx-\hy|^2\leq v(\hy)-v(\hx)\leq 2\|v\|_\infty,
\]
i.e.
\[
|\hx-\hy|\leq 2\sqrt{\ep\|v\|_\infty}. 
\]
Denoting by $\omega_v$ a continuity modulus for $v$, we derive 
\beq\label{eq optimal y}
\fr{1}{2\ep}|\hx-\hy|^2
\leq 
v(\hy)-v(\hx)
\leq 
\omega_v(|\hx-\hy|)
\leq
\omega_v\big( 2 \sqrt{\ep\|v\|_\infty}\big).
\eeq
Set $\psi(i)=|\xi|^2/{2\ep}$ for $\xi\in\R^d$. Note that 
\[
(x,y)\mapsto v(y)-\psi(x-y)-\varphi(x)
\]
takes a maximum at $(\hx,\hy)$.  It follows that
\beq\label{eq3.1}
D\varphi(\hx)=-D\psi(\hx-\hy)=\fr{1}{\ep}(\hy-\hx). 
\eeq
We choose $\gth\in \C^2(\R)$ so that 
\[
\gth(r)=r \ \ \ \FOR r\leq 1,\quad \gth(r) =2 \ \ \ \FOR r\geq 3, \quad\AND\quad \gth''(r)\leq 0 \ \ \ 
\FOR r\in\R.
\] Note that $0\leq \gth'(r) \leq 1$ and $\gth(r)\leq r$ for $r\in \R$.
For $M>0$, set $\gth_M(r)=M\gth(r/M)$ for $r\in\R$. Also, set $\psi_M(\xi)=\gth_M\circ \psi(\xi)$ 
for $\xi\in\R^d$. Note that 
\[
\psi_M(\xi)=\fr{1}{2\ep}|\xi|^2 \ \ \ \IF |\xi|\leq\sqrt{2\ep M},
\]
and 
\[\bald
&0\leq \psi_M(\xi)\leq 2M,\quad
D\psi_M(\xi)=\gth_M'\circ\psi(\xi)D\psi(\xi),
\\& D^2\psi_M(\xi)
=\gth_M''\circ\psi(\xi)D\psi(\xi)\otimes D\psi(\xi)
+\gth_M'\circ \psi(\xi)D^2\psi(\xi). 
\eald
\]
In particular, $\psi_M\in C_{\mrm{b}}^2(\R^d)$. 
Choosing $M>0$ large enough, we keep the property that 
\[
(x,y)\mapsto v (y)-\psi_M(x-y)-\varphi(x) \ \ \text{ takes a local maximum at }(\hx,\hy), 
\] 
as well as \ 
$\psi_M(\xi)=\psi(\xi)$ 
\ near $\xi=\hx-\hy$.
This last property, together with \eqr{eq3.1}, yields
\beq\label{eq3.2}
D\varphi(\hx)=D\psi_M(\hy-\hx)=\fr{\hy-\hx}{\ep}. 
\eeq
Since $\psi_M\in \C_{{b}}^2(\R^d)$, in view of Theorem \ref{thm sub equivalent} 
the subsolution property of $v$ reads as 
\[
\gl v (\hy)+H(\hy,D\varphi(\hx))\leq \II(v ,D\varphi(\hx),\hy). 
\]
By taking into account the fact that $v^\eps(\hx)\leq v(\hy)$, \eqref{eq3.1} and assumption {\bf (H1)}, we derive 
\beq\label{eq3.3}
\gl v^\eps (\hx)+H(\hx,D\varphi(\hx))\leq \II(v ,D\varphi(\hx),\hy)
+\omega_H\big(|\hx-\hy|(1+{|\hx-\hy|}/{\eps})\big).
\eeq
Observe that 
\begin{flalign*}
v (\hy+j(\hy,z))&-v (\hy)
=\Big( 
v (\hy+j(\hy,z))-\psi\big(\hx+j(\hx,z)
-\hy-j(\hy,z)\big)
\Big)
-
\Big(
v (\hy)-\psi(\hx-\hy)
\Big)
\\
&\qquad\qquad\quad
+\psi\big(\hx+j(\hx,z)-\hy-j(\hy,z)\big) -\psi(\hx-\hy)
\\
&\leq 
v^\ep(\hx+j(\hx,z))-v^\ep(\hx)
+\psi\big(\hx+j(\hx,z)-\hy-j(\hy,z)\big) -\psi(\hx-\hy)
\\
&=
v^\ep\big(\hx+j(\hx,z)\big)-v^\ep(\hx)
+\fr{1}{\ep}
\lan \hx-\hy, j(\hx,z)-j(\hy,z)\ran +\fr{1}{2\ep}|j(\hx,z)-j(\hy,z)|^2. 
\end{flalign*}
Let us fix $R>1$. Setting $E_1:=B$, $E_2:=B_R\setminus B$, $E_3:=\R^d\setminus B_R$, we have  
\begin{flalign*}
\II(v,D\varphi(\hx),\hy)
=
\sum_{i=1}^3 \int_{E_i} 
\Big(v (\hy+j(\hy,z))-v (\hy)-\1_B(z)\lan D\varphi(\hx),j(\hy,z)\ran \Big)
\nu(dz)
=:I_1+I_2+I_3
\end{flalign*}
For $I_3$, we have the following estimate:
\[\bald
I_3 
&=
\int_{\R^d\setminus B_R}
\Big(v^\ep(\hx+j(\hx,z))-v^\ep(\hx) \Big)\, \nu(dz) 
\\
&\quad 
+\int_{\R^d\setminus B_R}
\Big(
v (\hy+j(\hy,z))-v (\hy)-v^\ep(\hx+j(\hx,z))+v^\ep(\hx)
\Big)\,
\nu(dz)
\\
&\leq 
\int_{\R^d\setminus B_R}
\Big(
v^\ep(\hx+j(\hx,z))-v^\ep(\hx)
\Big)
\,\nu(dz)
+
4\|v\|_\infty\, \nu(\R^d\setminus B_R).
\eald
\]
For the integral $I_2$, we have:
\[\bald
I_2
&=
\int_{B_R\stm B}
\Big(
v (\hy+j(\hy,z))-v (\hy)
\Big)\,
\nu(dz)
\\
&\leq 
\int_{B_R\stm B}
\Big(
v^\ep(\hx+j(\hx,z))-v^\ep(\hx)
\Big)\,
\nu(dz)
\\
&\quad+
\int_{B_R\stm B}
\Big(
\fr{1}{\ep}\lan \hx-\hy, j(\hx,z)-j(\hy,z)\ran +\fr{1}{2\ep}|j(\hx,z)-j(\hy,z)|^2
\Big)
\,\nu(dz)
\\
&\leq 
\int_{B_R\stm B} \Big(v^\ep(\hx+j(\hx,z))-v^\ep(\hx)\Big)\,\nu(dz)
+
RC_jC_\nu\left(1+RC_j\right) \frac{|\hx-\hy|^2}{\eps} \ \ \ \text{by {\bf (N),\,(J1)}}.
\eald
\]
Last, the integral $I_1$ can be estimated as follows:
\[\bald 
I_1
&=
\int_B
\Big(v (\hy+j(\hy,z))-v (\hy)-\lan \underbrace{D\varphi(\hx)}_{(\hy-\hx)/\eps},j(\hy,z)\ran \Big)
\nu(dz)
\\&
\leq 
\int_B
\Big(
v^\ep(\hx+j(\hx,z)) -v^\ep(\hx) 
+\fr{1}{\ep}\lan \hx-\hy, j(\hx,z)-j(\hy,z)\ran +\fr{1}{2\ep}|j(\hx,z)-j(\hy,z)|^2
\\
&\quad +
\fr{1}{\ep}
\lan\hx-\hy,j(\hy,z)\ran
\Big)\,
\nu(dz)
\\
&=
\int_B
\Big(
v^\ep(\hx+j(\hx,z)) -v^\ep(\hx) 
+
\fr{1}{\ep}
\lan \hx-\hy,  j(\hx,z)\ran +\fr{1}{2\ep}|j(\hx,z)-j(\hy,z)|^2
\Big)\,
 \nu(dz)
\\
&\leq 
\int_B
\Big(
v^\ep(\hx+j(\hx,z)) -v^\ep(\hx) 
-\lan D\varphi(\hx), j(\hx,z)\ran
\Big)
\,\nu(dz) +
{C_\nu C_j^2}\frac{|\hx-\hy|^2}{2\ep}   \ \ \ \text{by {\bf (N),\,(J1)}}.
\eald
\]
In view of \eqref{eq optimal y}, we infer that there exists a positive constant 
$C(R)=C(C_j,C_\nu,R)$ such that 
\beq\label{eq sup convolution estimate}
\bald
\II (v,D\varphi(\hx),\hy)
&\leq  
\int_{\R^d}(v^\ep(\hx+j(\hx,z))-v^\ep(\hx)-\1_B(z)\lan D\varphi(\hx), j(\hx,z)\ran)\nu(dz)
\\&
+4\|v\|_\infty \nu(\R^d\setminus B_R)+C(R)\,\omega_v\big( 2 \sqrt{\ep\|v\|_\infty}\big)
\\&
\leq 
\II(v^\ep,D\varphi(\hx),\hx) +4\|v\|_\infty \nu(\R^d\setminus B_R)
+
C(R)\,\omega_v\big( 2 \sqrt{\ep\|v\|_\infty}\big). 
\eald
\eeq
Let us fix $\delta>0$. In view of assumption {\bf (N)}, we first choose $R>1$ big enough such that 
\[
4\|v\|_\infty\, \nu(\R^d\setminus B_R)<\delta/2.
\]
Next, we choose $\eps=\eps(\delta)>0$ small enough so that $\|v-v^\eps\|_\infty<\delta$,
\[
|\hx-\hy|(1+{|\hx-\hy|}/{\eps})
\leq
\sqrt{\eps}+2\omega_v\big(2 \sqrt{\ep\|v\|_\infty}\big)
\qquad
\hbox{in view of  \eqref{eq optimal y}},
\]
and 
\[
\omega_H\left(\sqrt{\eps}+2\omega_v\big(2 \sqrt{\ep\|v\|_\infty}\big)\right)
+
C(R)\,\omega_v\big( 2 \sqrt{\ep\|v\|_\infty}\big)
<
\frac{\delta}{2}.
\]
By going back to the inequality  \eqref{eq3.3} with this choice of $\eps>0$, we finally get, in view of  
\eqref{eq sup convolution estimate}
\[
\gl v^\eps (\hx)+H(\hx,D\varphi(\hx))\leq \II(v^\eps,D\varphi(\hx),\hx)
+\delta
\qquad
\IN \M.
\]
This proves the assertion, in view of Theorem \ref{thm sub equivalent}.  
\eproof

\section{Smoothing out subsolutions}\label{app.smooth subsolutions}

This section is aimed at showing that a Lipschitz and semiconvex subsolution of 
an equation of the form 
\beq\label{eq HJ a}
\gl u-\II u+H (x,Du)=a\qquad \IN \M,
\eeq
with $\lambda\geq 0$, $a\in\R$ and $H$ satisfying  {\bf (H2)}, can be approximated through a family of smooth functions that are subsolutions to \eqref{eq HJ a}, up to a small error. The nonlocal operator $\II$ is of the form \eqref{def I}, where we still assume conditions {\bf (N),\,(J1)}, but we need to reinforce condition  {\bf (J2)} in favor of  the following:
\begin{itemize}
\item[\bf (J3)] There is a constant $C_0>0$ such that 
\[
\int_{\R^d}|j(x,z)-j(y,z)|\nu(dz)\leq C_0|x-y| \ \ \ \FOR x,y\in\R^d. 
\]
\end{itemize}
The precise statement is the following. 

\begin{proposition}\label{prop convolution} Assume {\bf (N),\,(J1),\,(J3)} and {\bf (H2)}. Let $v\in\Lip(\T^d)$ be a semiconvex subsolution of \eqref{eq HJ a}
for some $\lambda\geq 0$ and $a\in\R$. 
Then, for each $\delta>0$, there exists a subsolution $u_\delta\in \C^\infty(\M)$ of \eqr{eq HJ a} with $a+\delta$ in place of $a$ satisfying  $\|u_\delta-v\|_\infty <\delta$. 
\end{proposition}

\begin{remark}
We do not know whether the above statement keeps holding when condition {\bf (J3)} is replaced by the weaker condition {\bf (J2)}. 
\end{remark}

By merging Propositions \ref{prop sup convolution} and \ref{prop convolution}, we readily obtain the following result.  

\begin{theorem}\label{thm smooth approximation} 
Assume {\bf (N),\,(J1),\,(J3)} and {\bf (H1)-(H2)}. Let $u\in \C(\T^d)$ be a subsolution to  
\eqref{eq0} for some $\lambda\geq 0$. Then there exist a sequence $(u_n)_n$ 
in $\C^\infty(\M)$ uniformly converging to $u$ on $\M$ and satisfying 
\[
\gl u_n(x)-\II u_n(x)+H (x,Du_n(x))\leq \frac1n\qquad \FORALL x\in\M.
\]
\end{theorem}

\begin{remark}
Beside allowing us to simplify some proofs, the interest of having such an approximation result relies on the fact that the minimization problems appearing in Theorem \ref{thm.discounted Mather measures} and, primarily, in Theorem \ref{thm.Mather measures 1} can be performed over the class of closed measures, in analogy to what is done in the context of weak KAM Theory for $\gl=0$, see for instance \cite[Theorem 5.7]{DFIZ1}. We will elaborate this further in this section, see in particular Proposition \ref{prop closed measures minimization}. 
\end{remark}

\bproof[Proof of Proposition \ref{prop convolution}]
Up to replacing $H$ with $H-a$, we can assume, without loss of generality, that $a=0$. 
Let $\rho_\eps$ be a standard mollification kernel, with $\supp \rho_\eps \subset B_\eps$.  
Let 
\[
{I }v(x,\xi)=\int_{\R^d}(v(\xi+j(x,z))-v(\xi)-1_B(z)\du{Dv(\xi),j(x,z)})\nu(dz),
\]
defined for  $(x,\xi)\in\R^d\tim\R^d$. A crucial point for the proof of this proposition 
is to show that
\beq\bald \label{s7.2}
\int_{\R^d} {I }v(x,x-y)\rho_\eps(y)dy&=
\int_{\R^d}(v_\eps(x+j(x,z))-v_\eps(x)-1_B(z)\du{Dv_\eps(x),j(x,z)})\nu(dz)
\\&=\II  v_\eps(x),
\eald
\eeq
where 
\[
v_\eps(x):=(v*\rho_\eps)(x)=\int_{\R^d}v(x-y)\rho_\eps(y)dy.
\]
Clearly, $v_\eps\in \C^\infty(\T^d)$. 

To check \eqr{s7.2}, we fix $x\in\R^d$. Observe that 
\[
f: (\xi,z)\mapsto v(\xi+j(x,z))-v(\xi)-\1_B(z)\du{Dv(\xi),j(x,z)}
\]
is measurable with respect to the $\sigma$-algebra 
$\mathcal{L}(\R^d)\otimes \mathcal{B}(\R^d)$ (where $\mathcal{L}(\R^d)$ stands for the Lebesgue 
$\sigma$-algebra on $\R^d$).  
If $|z|\geq 1$, then
\[
f(\xi,z)=v(\xi+j(x,z))-v(\xi)\geq -2\|v\|_\infty.
\]
If $|z|<1$, then, since 
\[
v(\xi+\eta)-v(\xi)\geq \du{Dv(\xi), \eta}-C_v|\eta|^2
\qquad\FORALL \xi,\eta\in\R^d
\]
for some constant $C_v\geq 0$ by semiconvexity of $v$, we have
\[
f(\xi,z)\geq -C_v|j(x,z)|^2.
\]
If we set 
\[
g(z)=2\|v\|_\infty\1_{\R^d\setminus B}(z)+ C_v|j(x,z)|^2\1_B(z),
\]
we infer from the above inequalities that $f(\xi,z)+g(z)\geq 0$ in $\R^d\tim\R^d$. 
By Tonelli's Theorem, we have
\begin{flalign*}
\int_{B_\eps\times\R^d} \Big(f(x-y,z)+g(z) \Big)\, \rho_\eps(y) dy&\otimes\nu(dz)
=\int_{B_\eps}\left(\int_{\R^d}(f(x-y,z)+g(z))\nu(dz)\right)\rho_\eps(y) dy
\\
&=
\int_{\R^d}\left(\int_{B_\eps}(f(x-y,z)+g(z))\rho_\eps(y) dy\right)\nu(dz)
=:I (x),
\end{flalign*} 
where the double integral appearing above makes sense as a nonnegative number. 
We note that 
\[
I (x)=\int_{\R^d}\Big (v_\eps(x+j(x,z))-v_\eps(x)-\1_B(z)\du{Dv_\eps(x),j(x,z)}
+g(z) \Big)\nu(dz).
\]
Since $g\in L^1(\R^d; \nu)$ by assumptions {\bf (N)} and {\bf (J1)}, we furthermore have 
\[
\int_{B_\eps}\int_{\R^d}g(z)\nu(dz)\rho_\eps(y)dy=\int_{\R^d}\int_{B_\eps}g(z)\rho_\eps(y)dy\nu(dz),
\]
and hence
\[\bald
\int_{B_\eps}&\left(\int_{\R^d}f(x-y,z)\nu(dz)\right)\rho_\eps(y)dy 
\\&=\int_{\R^d}\Big(v_\eps(x+j(x,z)-v_\eps(x)-\1_B(z)\du{Dv_\eps(x),j(x,z)}\Big)\,\nu(dz)=\II v_\eps(x). 
\eald
\]
Since $v$ is a subsolution to \eqr{eq HJ a} and is semiconvex, by Proposition \ref{prop pointwise subsol} we have
\[
\gl v+H(x,Dv(x))\leq \II v(x)\qquad \text{for a.e.}\  x\in \R^d.  
\] 
For $x\in\R^d$ and $y\in B_\eps$, we observe, by {\bf (J3)}, that
\[\bald
\II v(x-y)-I v(x,x-y)
&=
\int_{\R^d}
\Big(
v(x-y+j(x-y,z))-v(x-y+j(x,z))
\\
&\quad -
\1_B(z)\du{Dv(x-y),j(x-y,z)-j(x,z)}\Big)\,
\nu(dz)
\\&\leq 2\int_{\R^d}\|Dv\|_\infty|j(x-y,z)-j(x,z)|\nu(dz)
\\&\leq 2C_0\|Dv\|_\infty|y|\leq 2C_0\|Dv\|_\infty \eps 
\eald\]
If we set
\[
\go(r)=\sup\{|H(x,p)-H(y,p)|\mid |p|\leq \|Dv\|_\infty, |x-y|\leq r\},
\]
then, by \eqr{s7.2},
\[
\bald
\gl v_\eps(x)+H(x,Dv_\eps(x))&\leq \go(\eps)+\int_{B_\eps}I v(x,x-y)\rho_\eps(y)dy+2C_0\|Dv\|_\infty\eps
\\&= \II v_\eps(x)+\go(\eps)+2C_0\|Dv\|_\infty \eps. 
\eald
\]
Let $\delta>0$ and choose $\eps=\eps(\delta)>0$ such that 
\[
\go(\eps)+2C_0\|Dv\|_\infty \eps<\delta
\qquad
\hbox{and}
\qquad
\|v-v_\eps\|_\infty<\delta.
\]
The assertion follows by setting $u_\delta:=v_{\eps(\delta)}$. 
\eproof

Let $Q$ be a compact subset of $\R^d$ and set $K:=\M\times Q$. We proceed to explain how, with the aid of Theorem \ref{thm smooth approximation}, the approach that was taken in Section \ref{sec.duality} could be modified by minimizing the integral of $\phi\in\C(K)$ over the class of closed measures. The definition of 
closed measure on $K$ associated with $(z,\gl)\in \T^d\tim[0,+\infty )$ is the following.

\begin{definition}\label{def closed measure I}
Let  $(z,\gl)\in \T^d\tim[0,+\infty )$ be fixed.
A {\em closed measure} $\mu$ on $K$ associated with $(z,\gl)$ is a Borel probability measure on $K$ satisfying  
\[
\int_{K} \bigg(\gl \psi(x)-\II \psi(x)+\langle \xi, D\psi(x)\rangle \bigg)\mu(dx d\xi)
=\gl \psi(z)\quad\hbox{for all $\psi\in \C^2(\T^d)$}. 
\]
We denote by $\frak{C}_K(z,\gl)$ the space of all closed measures on $K$ associated with $(z,\gl)$.  Note that, when $\gl=0$, this set is actually independent of $z$. 
\end{definition}

The relation between closed measures and the probability measures $\cG'(z,\gl)$ introduced in Section \ref{sec.duality} by duality is the following. 

\begin{lemma}\label{lemma closed measure} 
Let $\lambda \geq 0$ and $z\in\M$. Then $\cG'(z,\gl)\subseteq \frak{C}_K(z,\gl)$.
\end{lemma}

\bproof
Let us pick $\psi\in\C^2(\M)$ and set 
\[
\phi(x,\xi):=\gl \psi(x)-\II\psi(x)+\langle \xi,D\psi(x)\rangle
\qquad
\FORALL (x,\xi)\in K. 
\] 
It is easily checked that $(\phi,\psi)$ and $(-\phi,-\psi)$ both belong to $\cF(\gl)$. For every $\mu\in\cG'(z,\gl)$ we thus have $\langle \mu, \phi \rangle\geq \gl\psi(z)$ and $\langle \mu, -\phi \rangle\geq -\gl\psi(z)$, yielding $\langle \mu, \phi \rangle= \gl\psi(z)$. 
\eproof

The following holds. 

\begin{lemma}\label{lem.ma25} Let $\lambda \geq 0$, $z\in\M$ and $(\phi,u)\in\cF (\gl)$. For any 
$\mu\in\frak{C}_K(z,\gl)$ we have
\[ 
\gl u(z)\leq \int_K \phi(x,\xi)\mu(dxd\xi). 
\]
\end{lemma}

\bproof According to Theorem \ref{thm smooth approximation}, there exists a sequence $(u_n)_{n\in\N}$ in $
\C^2(\T^d)$ uniformly converging to $u$ on $\M$  such that 
\[
\gl u_n(x)-\II u_n(x)+H_\phi(x,Du_n(x))\leq 1/n\ \ \ \FORALL x\in\R^d.
\]
Integrating with respect to a measure $\mu\in\frak{C}_K(z,\gl)$ the inequality
\[
\gl u_n(x)-\II u_n(x)+\xi\cdot Du_n(x)-\phi(x,\xi)\leq 1/n\ \ \ \FORALL (x,\xi)\in K=\R^d\tim Q,
\]
we obtain
\[\bald
1/n&\geq -\int_K \phi(x,\xi)\,\mu(dxd\xi)+\int_K(\gl u_n(x)-\II u_n(x)+\xi\cdot Du_n(x)) \,\mu(dxd\xi)
\\&=-\int_K \phi(x,\xi)\,\mu(dxd\xi)+\gl u_n(z).
\eald
\]
By sending $n\to +\infty$ we get
\[
\gl u(z)\leq \int_K\phi(x,\xi)\mu(dxd\xi).\qedhere 
\]
\eproof
\medskip

In view of Lemmas \ref{lemma closed measure} and \ref{lem.ma25} and of Theorems \ref{thm discounted Mather measures} and \ref{thm Mather measures} we get the following result. 

\begin{proposition}\label{prop closed measures minimization}
For every $\lambda \geq 0$ and $z\in\M$, the following holds:
\[
\min_{\mu\in{\cGprime}(z,\lambda)}\int_K \phi(x,\xi)\,\mu(dxd\xi)
=
\min_{\mu\in \frak{C}_K(z,\gl)}\int_K \phi(x,\xi)\,\mu(dxd\xi).
\]
\end{proposition}
\end{appendix}

\bibliography{discount}
\bibliographystyle{siam}

\bye